\documentclass[12pt, reqno]{amsart}
\usepackage{amssymb, hyperref}
\usepackage{amsthm}
\usepackage{amsmath}
\usepackage{graphicx}
\usepackage[all,2cell]{xy}
\usepackage{verbatim}
\usepackage{tikz}
\usepackage{tikz-cd}
\usepackage{hyperref}
\usepackage{array}
\usetikzlibrary{matrix,arrows,decorations.pathmorphing}
\usepackage{mathrsfs}
\numberwithin{equation}{section}
\newtheorem*{theorem*}{Theorem}
\newtheorem*{corollary*}{\bf Corollary}
\newtheorem*{remark*}{\bf Remark}
\usepackage{lipsum}
\usepackage{lineno}
\newtheorem{theorem}{Theorem}[section]

\newtheorem{lemma}[theorem]{Lemma}

\title[Torus quotient of Richardson varieties]
{Torus quotient of Richardson varieties in Orthogonal and Symplectic Grassmannians}
\newtheorem{proposition}[theorem]{Proposition}

\title[Torus quotient of Richardson varieties]
{Torus quotient of Richardson varieties in Orthogonal and Symplectic Grassmannians}
 \author{Arpita Nayek}
\address{%
Arpita Nayek\\
Department of Mathematics and Statistics \\
Indian Institute of Technology, Kanpur\\
Kanpur-208016 \\
India\\
Email: anayek@iitk.ac.in\\
}
\author{S.K. Pattanayak}
\address{%
S.K. Pattanayak\\
Department of Mathematics and Statistic\\
Indian Institute of Technology, Kanpur\\
Kanpur-208016\\
India\\
Email:santosha@iitk.ac.in\\
}


\begin{document}
\maketitle
\begin{abstract}
For any simple, simply connected algebraic group $G$ of type $B,C$ and $D$ and for any maximal parabolic subgroup $P$ of $G$, we provide a criterion for a Richardson variety in $G/P$ to admit semistable points for the action of a maximal torus $T$ with respect to an ample line bundle on $G/P$. 
\end{abstract}



\section{Introduction}\label{s.Introduction}

For the action of a maximal torus $T$ on the Grassmannian $G_{r,n}$, the GIT quotients have been studied by several authors. In \cite{HK} Hausmann and Knutson identified the GIT quotient of the Grassmannian $G_{2,n}$ by the natural action of 
the maximal torus with the moduli space of polygons in $\mathbb R^3$ and this GIT quotient can also be realized as 
the GIT quotient of an $n$-fold product of projective lines by the diagonal action of $PSL(2, \mathbb C)$. In the 
symplectic geometry literature these spaces are known as polygon spaces as they parameterize the $n$-sides polygons 
in $\mathbb R^3$ with fixed edge length up to rotation. More generally, $G_{r,n}//T$ can be identified with the GIT 
quotient of $(\mathbb P^{r-1})^{n}$ by the diagonal action of $PSL(r, \mathbb C)$ called the Gelfand-MacPherson 
correspondence. In \cite{Kap1} and \cite{Kap2} Kapranov studied the Chow quotient of the Gassmannians 
and he showed that the Grothendieck-Knudsen moduli space $\overline{M}_{0,n}$ of stable $n$-pointed curves of genus zero 
arises as the Chow quotient of the maximal torus action on the Grassmannian $G_{2,n}$.

Let $G$ be a simply connected semi-simple algebraic group over an algebraically closed field $K$. Let $T$ be a maximal torus of $G$ and $B$ be a Borel subgroup of $G$ containing $T$. In \cite{kannan1} and \cite{kannan2}, the parabolic subgroups $P$ of $G$ containing $B$ are described for which there exists an ample line bundle $\mathcal{L}$ on $G/P$ such that the semistable points $(G/P)^{ss}_{T}(\mathcal{L})$ are the same as the stable points $(G/P)^{s}_{T}(\mathcal{L})$. In \cite{STR} Strickland reproved these results. 

In \cite{KS}, when $G$ is of type $A$, $P$ is a maximal parabolic subgroup of $G$ and $\mathcal{L}$ is the ample generator of the Picard group of $G/P$, it is shown that there exists unique minimal Schubert variety $X(w)$ admitting semistable points with respect to $\mathcal{L}$. For other types of classical groups the minimal Schubert varieties admitting semistable points were described in \cite{KP} and \cite{santosh}.

A Richardson variety $X_{w}^{v}$ in $G/P$ is the intersection of the Schubert variety $X_{w}$ in $G/P$ with the opposite Schubert variety $X^{v}$ therein. For $G=SL_n$ and $P$ a maximal parabolic in $G$ a criterion for the Richardson varieties in $G/P$ to have nonempty semistable locus with respect to an ample line bundle $\mathcal{L}$ on $G/P$ is given in \cite{KPPU}. In this paper, we give a criterion for a Richardson variety in $G/P$ to have nonempty semistable locus with respect to the action of a maximal torus $T$ on $G/P$, where $G$ is of type $B$, $C$ and $D$ and $P$ is a maximal parabolic subgroup in $G$.

The organisation of the paper is as follows. Section 2 consists of preliminary notions and some terminologies from algebraic groups and Geometric invariant theory. Section 3 gives a necessary condition for a Richardson variety to admit a semistable point. In section 4 we give a sufficient condition for a Richardson variety in type $B$ and $C$ to admit a semistable point and in section 5 a sufficient condition is given for type $D$.  

\section{Preliminaries and notation}\label{s.Prel-notation}
 In this section, we set up some notation and preliminaries. We refer to \cite{carter}, \cite{hum1}, \cite{hum2} and \cite{spr} for preliminaries in Lie algebras and algebraic groups. Let $G$ be a semi-simple algebraic group over an algebraically closed field $K$. We fix a maximal torus $T$ of $G$ and a Borel subgroup $B$ of $G$ containing $T$. Let $U$ be the unipotent radical of $B$. Let $N_{G}(T)$ (respectively, $W=N_{G}(T)/T$)  be the normalizer of $T$ in $G$  (respectively, the Weyl group of $G$ with respect to $T$).  Let $B^{-}$ be  the Borel subgroup of $G$ opposite to $B$ determined by $T$. We denote by $R$ the set of roots with respect to $T$ and we denote by $R^{+}$ the set of positive roots with respect to $B$. Let $U_{\alpha}$ denote the one-dimensional $T$-stable subgroup of $G$ corresponding to the root $\alpha$ and we denote $U_{\alpha}^*$ by the open set $U_{\alpha} \setminus \{identity\}$. Let $S=\{\alpha_1,\ldots,\alpha_l\}\subseteq R^{+}$ denote the set of simple roots and for a subset $I\subseteq S$ we denote by $P_I$  the parabolic subgroup of $G$ generated by $B$ and $\{n_{\alpha}: \alpha \in I^{c}\}$, where $n_{\alpha}$ is a representative of $s_{\alpha}$ in $N_{G}(T)$. Let $W^{I}=\{w\in W: w(\alpha)\in R^{+} \,\,  \mbox{for each} \,\, \alpha \in I^c\}$ and $W_{I}$ be the subgroup of $W$ generated by the simple reflections $s_{\alpha}$, $\alpha\in I^c$. Then every $w\in W$ can be uniquely expressed as $w=w^{I}.w_{I}$, with $w^{I}\in W^{I}$ and $w_{I}\in W_{I}$. Denote by $w_{0}$ the longest element of $W$ with respect to $S$. Let $X(T)$ (respectively,  $Y(T)$) denote the  group of all  characters of $T$ ( respectively,  one-parameter subgroups of $T$ ). Let $E_{1}:=X(T)\otimes\mathbb{R}$  and $E_{2}=Y(T)\otimes\mathbb{R}$. Let $\langle .,.\rangle:E_{1}\times E_{2}\rightarrow\mathbb{R}$ be the canonical non-degenerate bilinear form. Let $\{ \lambda_{j}: j=1,2, \cdots l \}\subset  E_{2}$  be the basis of $E_{2}$ dual to $S$. That is, $\langle\alpha_{i},\lambda_{j}\rangle=\delta_{ij}$ for all $1\leq i, ~ j\leq l$. Let $\bar{C}:=\{\lambda\in E_{2}|\langle\alpha, \lambda \rangle\geq 0 \,\, \forall \,\, \alpha\in R^{+}\}$. Note that for each $\alpha\in R$, there is a homomorphism $SL_{2}\xrightarrow{\phi_{\alpha}}G$ (see [2, p.19 ] ). We have $\check{\alpha}:G_{m}\rightarrow G$ defined by $\check{\alpha}(t)=\phi_{\alpha}(\begin{pmatrix}t & 0\\0 & t^{-1}\end{pmatrix})$. We also have $s_{\alpha}(\chi)=\chi-\langle\chi,\check{\alpha}\rangle\alpha$ for all $\alpha\in R$ and $\chi\in E_{1}$. Set $s_{i}=s_{\alpha_{i}}$ for every $i=1,2,\ldots,l$. Let $\{\omega_{i}:i=1,2,\ldots,l\}\subset E_{1}$ be the fundamental weights; i.e. $\langle\omega_{i},\check{\alpha_{j}}\rangle=\delta_{ij}$ for all $i,j=1,2,\ldots,l$. 

 Let $X_{w}=\overline{BwB/B}$  ( respectively, $X^{v}=\overline{B^{-}vB/B}$ ) denote the Schubert variety corresponding to $w$  (respectively, the opposite Schubert variety corresponding to $v$ ). Let $X_{w}^{v}:=\overline{BwB/B}\cap\overline{B^{-}vB/B}$ denote the Richardson variety corresponding to $v$ and $w$ where $v\leq w$ in the Bruhat order. Such varieties were first considered by Richardson in \cite{R}, who shows that such intersections are reduced and irreducible whereas the cell intersection $C_w \cap C^v$ have been studied by Deodhar \cite{D}. Richardson varieties have shown up in several contexts: such double coset intersections $BwB \cap B^-vB$ first appear in \cite{KL1}, \cite{KL2} and their standard monomial theory is studied in \cite{LL} and \cite{BL}. We refer to \cite{LR} for preliminaries in standard monomial theory. 

We recall the definition of the Hilbert-Mumford numerical function and the definition of semistable points from \cite{MFK}. We refer to \cite{New} for notations in geometric invariant theory. 

Let $X$ be a projective variety with an action of a reductive group $G$. A point $x \in X$ is said to be semi-stable with respect to a $G$-linearized 
line bundle $\mathcal L$ if there is a positive integer $m \in \mathbb N$, and 
a $G$-invariant section $s \in H^0(X, \mathcal L^m)^G$ with $s(x) \neq 0$.

Let $\lambda$ be a one-parameter subgroup of $G$. Let $x \in \mathbb 
P(H^0(X,\mathcal L)^*)$ and $\hat{x}= \sum_{i=1}^rv_i$, where each $v_i$ is a 
weight vector of $\lambda$ of weight $m_i$. Then the Hilbert-Mumford numerical function is defined by \[\mu^{\mathcal L}(x, \lambda):=-min\{m_i: i =1, \cdots, r\}\] Then the  Hilbert-Mumford criterion says that $x$ is semistable if and only if $\mu^{\mathcal L}(x, \lambda) \geq 0$ for all one parameter subgroup $\lambda$. 

We recall the following result from \cite{CSS} which will be used in section 3.

\begin{lemma}
 Let $G$ be a semisimple algebraic group, $T$ be a maximal torus, $B$ be a Borel subgroup of $G$ containing $T$ and $\overline{C}$ be as 
 defined above. 
 
 (a) Let $\mathcal L$ be a line bundle defined by the character $\chi \in X(T)$. Then if $x \in G/B$ is represented by $bwB$, $b \in B$ and $w \in W$ is represented by 
 an element of $N$ in the Bruhat decompsition of $G$ and $\lambda$ is a one parameter subgroup of $T$ which lies in $\overline{C}$, we have 
 $\mu^{\mathcal L}(x, \lambda)=-\langle w(\chi), \lambda \rangle$.
 
 (b) Given any set $S$ of finite number of one parameter subgroup $\lambda$ of $T$, there is an ample line bundle $\mathcal L$ on $G/B$ such that $\mu^{\mathcal L}(x, \lambda)
 \neq 0$ for all $x \in G/B$, $\lambda \in S$.  
 
 \end{lemma}

In this paper, we present results for Richardson varieties in the orthogonal and symplectic Grassmannians. For any character $\chi$ of $B$, we denote by $\mathcal{L}_{\chi}$, the line bundle on $G/B$ given by the character $\chi$. We denote by $(X^{v}_{w})_{T}^{ss}(\mathcal{L}_{\chi})$ the semistable points of $X_{w}^{v}$ for the action of $T$ with respect to the line bundle $\mathcal{L}_{\chi}$. Using the notations from \cite{BL} we recall the following theorem which is needed in the proofs of the main theorems. 


\begin{theorem}[\cite{BL}, Proposition 6.]
  Let $\lambda$ be a dominant weight. The restriction to $X^v_w$ of the $p_{\pi}$, where $v \leq e(\pi) \leq i(\pi) \leq w$ form a basis of $H^0(X^v_w,\mathcal{L}_{\lambda})$.
\end{theorem}

In the rest of this section we recall Bruhat ordering in the Weyl groups of type $B$, $C$ and $D$ and how it is related to the Bruhat order for the symmetric group $S_n$. 

\textbf{Bruhat order for type $B_n$ or $C_n$:} We consider $\alpha_1$ as special node of Dynkin diagram for type $B$ or $C$. So as a set of generators of Weyl group, we take $S=\{s_1, s_2, \ldots, s_n\}$, where $s_1=(1,-1)$ and $s_i=(i-1,i)$ $\forall 2 \leq i \leq n$ as in \cite{BB}. 
 
 As in \cite{IF} we use a formula for computing the length of $\sigma \in W$ given by \begin{equation}
 	l_B(\sigma)=\frac{inv(\sigma)+neg(\sigma)}{2},
 \end{equation} 
 
 where $inv(\sigma)=|\{(i,j)\in[-n,n]\backslash\{0\}\times[-n,n]\backslash\{0\}:i<j,\sigma(i)>\sigma(j)\}|$ and $neg(\sigma)=|\{i\in[1,n]:\sigma(i)<0\}|$.
 
 The following result gives a combinatorial characterization of the Bruhat order in $B_n$.
 
 \begin{lemma}[\cite{IF}, Proposition 2.8]
 	Let $\sigma$, $\tau$ $\in W$. Then $\sigma\leq\tau$ in the Bruhat order of $B_n$ if and only if $\sigma\leq\tau$ in the Bruhat order of the symmetric group $S_{[-n,n]\backslash\{0\}}$ where $S_{[-n,n]\backslash\{0\}}$ is the permutation group of integers $-n, -(n-1),\ldots, -1, 1, \ldots, n-1,n$. 
 \end{lemma}
 
 \textbf{Bruhat order for type $D_n$:} As above we consider $\alpha_1$ as special node for Dynkin diagram for type $D$. For a set of generators of Weyl group we have $S=\{s_1, s_2, \ldots, s_n\}$, where $s_1=(1,-2)(-1,2)$ and $s_i=(i-1,i)$ $\forall 2 \leq i \leq n$ as in \cite{BB}. 
 
 As in \cite{IF} we use a formula for computing the length of $\sigma \in W$ given by 
 \begin{equation}
 	l_D(\sigma)=\frac{inv(\sigma)-neg(\sigma)}{2},
 \end{equation} 
 
 where $inv(\sigma)$ and $neg(\sigma)$ are as defined above.
 
 The following result gives a combinatorial characterization of the Bruhat order in $D_n$.
 
 \begin{lemma}[\cite{BB}, Theorem 8.2.8]
 	Let $\sigma$, $\tau$ $\in W$. Then $\sigma\leq\tau$ in the Bruhat order of $D_n$ if and only if 
 	
 	(i) $\sigma \leq_B \tau$ (Bruhat order in type $B$) and
 	
 	(ii) $\forall a, b \in [1,n]$, if $[-a, a] \times [-b, b]$ is an empty rectangle for both
 	$\sigma$ and $\tau$ and $\sigma[-a-1, b+1] = \tau[-a-1, b+1]$, then $\sigma[-1, b + 1] \equiv \tau[-1, b + 1]$ (\text{mod $2$}) where $\sigma[i,j]=|\{a\in[-n,n]: a \leq i \text{ and } \sigma(a) \geq j\}|$ for $i,j \in [-n,n]$. 
 \end{lemma}

\section{A necessary condition for admitting semi-stable points}\label{s.the-criterion}
Let $G$ be a simple simply-connected algebraic group and $P_r$ be a parabolic subgroup of $G$ corresponding to the simple root $\alpha_r$. Let $\mathcal L_r$ be the line bundle on $G/P_{r}$ corresponding to the fundamental weight $\omega_r$. In this section, we provide a criterion for Richardson varieties in $G/P_r$ to admit semistable points with respect to $\mathcal L_r$. This criterion was proved for type $A$ in \cite{KPPU}.  

\begin{proposition}
Let $G$ be a simple simply connected algebraic group and let $P_r$ be the maximal parabolic corresponding to the simple root $\alpha_r$. Let $\mathcal L_r$ be the line bundle on $G/P_{r}$ corresponding to the fundamental weight $\omega_r$. Let $v,w \in W^{P_r}$. If $(X^{v}_{w})_{T}^{ss}(\mathcal{L}_{r})\neq\emptyset$ then $v(n\omega_r) \geq 0$ and $w(n\omega_r) \leq 0$.
\end{proposition}


\begin{proof}
Let $\chi =n\omega_{r}$. Assume that  $(X^{v}_{w})_{T}^{ss}(\mathcal{L}_{\chi})\neq\emptyset$. Let \small{$x\in ((BwP_{r}/P_{r})\bigcap (B^{-}vP_{r}/P_{r}))_{T}^{ss}(\mathcal{L}_{\chi})$}. \normalsize Then by Hilbert-Mumford criterion \cite[Theorem 2.1]{MFK}, we have $\mu^{\mathcal{L}_{\chi}}(x,\lambda)\geq 0$ for all one parameter subgroups $\lambda$ of $T$. Since $x\in ((BwP_{r}/P_{r})\bigcap (B^{-}vP_{r}/P_{r}))_{T}^{ss}(\mathcal{L_{\chi}})$,  using \cite[Lemma 2.1]{CSS}, we see that $\mu^{\mathcal{L}_{\chi}}(x,\lambda)=-\langle w(\chi),\lambda\rangle$  for every one parameter subgroup $\lambda$ in the fundamental chamber associated to $B$, and  $\mu^{\mathcal{L}_{\chi}}(x,\lambda)=\mu^{\mathcal{L}_{\chi}}(w_{0}x,w_{0}\lambda w_{0}^{-1})=-\langle w_{0}v(\chi),w_{0}(\lambda)\rangle=-\langle v(\chi),\lambda\rangle$ for every one parameter subgroup $\lambda$ of $T$ in the Weyl  chamber associated to $B^{-}$. Since $x$ is a semistable point, we have $\mu^{\mathcal{L}_{\chi}}(x,\lambda)\geq 0$ for every one parameter subgroup $\lambda$ of $T$. Hence $\langle w(\chi),\lambda\rangle\leq 0$ for all $\lambda$ in the Weyl  chamber associated to $B$ and $\langle v(\chi),\lambda\rangle \geq 0$ for all $\lambda$ in the Weyl  chamber associated to $B^{-}$. This implies that $w(\chi)\leq 0$ and $v(\chi)\geq 0$. 
\end{proof}

For $G$ is of type $A$ in \cite{KPPU} it is shown that the above conditions are also sufficient. For type $B,C$ and $D$ the example below shows that the conditions $w(\chi)\leq 0$ and $v(\chi)\geq 0$ are only necessary but not sufficient. 

{\bf Example:} Let $G$ be either of type $B_4$ or $C_4$ and $\chi= \omega_3$. Let $v=(1,2,-3,4)=s_3s_2s_1s_2s_3$ and $w=(1,4,-3,2)=s_3s_4s_2s_1s_2s_3$. We have $v(\omega_3)=\alpha_4$ and $w(\omega_3)=-\alpha_3$.  The sections of $\mathcal{L}_{\chi}$ on $X^{v}_{w}$ are of the form $p_v^m p_w^n$ where $m,n \in \mathbb N$. But, $m v(\chi)+n w(\chi) \neq 0$ for any $m,n \in \mathbb N$. So these sections are not $T$-invariant. So the set $(X^{v}_{w})_{T}^{ss}(\mathcal{L}_{\chi})$ is empty. 

 If $G$ is of type $D_4$ and $\chi=\omega_3$, we take $v=(-1,4,-2,3)=s_4s_1s_2s_3$ and $w=(-1,2,-4,3)=s_4s_3s_1s_2s_3$. Here we have $v(\omega_3)=\alpha_3$ and $w(\omega_3)=-\alpha_4$. As in the last paragraph, here also we conlude that the set $(X^{v}_{w})_{T}^{ss}(\mathcal{L}_{\chi})$ is empty.

In order to find a sufficient condition for the Richardson varieties to admit semistable points we first need to classify all $v,w \in W^P$ satisfying the above conditions. Since $\chi$ is a dominant weight we have $w_1(\chi) \leq w_2(\chi)$ for $w_1 \geq w_2$. So we just need to describe all maximal $v$ and minimal $w$ such that $v(\chi)\geq 0$ and $w(\chi)\leq 0$. Note that for $G$ is of type $A$ since all the fundamental weights are minuscule the maximal $v$ and minimal $w$ satisfying the above conditions are unique (see \cite{KS}) but for other types this is not the case. 

We conclude this section by introducing some notation here:

{\bf Notation:} For $s,t \in \mathbb Z$ such that $s \leq t$ we set $[s,t] = \{s,s+1, \ldots,t\}$. For $p \in \mathbb{N}$ we set $J_{p,[s,t]} =\{(i_1,i_2,\ldots i_p): i_k \in [s,t], \forall k$ and $i_{k+1}-i_k \geq 2\}$. For $\underline{i} \in J_{p, [s, t]}$ we set $[\underline{i}] = \{i_1, i_2, \ldots, i_p\}$ and $-\underline{i} = (-i_p, -i_{p-1}, \ldots, -i_1) \in J_{p, [-t, -s]}$. For a set $S \subset \mathbb{Z}$, $S \uparrow$ denotes the integers in the set $S$ occurring in increasing order and $S \downarrow$ denotes the integers in the set $S$ occurring in decreasing order.

\section{Type B and C}

Now for $G$ is of type $B$ and $C$ and for a fundamental weight $\omega_r$, we are in a position to describe all the minimal $w \in W^{I_r}$ and maximal $v \in W^{I_r}$ such that $v(\omega_r) \geq 0$ and $w(\omega_r) \leq 0$.
  
 \begin{proposition}
   The set of all maximal $v$ in $W^{I_r}$ such that $v(\omega_r) \geq 0 $ for type $B_n$ and $C_n$ are the following: 

   (i) Let $r=1$. Then
   \[
    v=\left\{
    \begin{array}{lr}
    (-(n-1), -(n-3), \ldots, -3, -1, 2, 4, \ldots, n-2, n),& \text{if n is even}\\
     (-(n-1), -(n-3), \ldots, -4, -2, 1, 3, 5, \ldots, n-2, n),& \text{if n is odd}.\\
    \end{array}
    \right.
    \] 
    
   
   (ii) Let $2 \leq r \leq n-1$ and $(n+1)-r=2m$. For any $\underline{i} = (i_1,i_2, \ldots, i_m) \in J_{m,[2,n]}$, there exists unique maximal $v_{\underline{i}}\in W^{I_r}$ such that $v_{\underline{i}}(\omega_r)=(\sum_{k=1}^m\alpha_{i_k})$. We have 
   $v_{\underline{i}}=([1,n]\backslash \{[\underline{i}], [\underline{i^{\prime}}]\}\uparrow, -\underline{i^{\prime}}, \underline{i})$, where $\underline{i^{\prime}} = (i_1-1, i_2-1, \ldots, i_m-1) \in J_{m, [1, n-1]}$.

    (iii) Let $2 \leq r \leq n-1$ and $(n+1)-r=2m+1$. For any $\underline{i}=(i_1,i_2,\ldots , i_m) \in J_{m,[3, n]}$, there exists unique maximal $v_{\underline{i}} \in W^{I_r}$ such that $v_{\underline{i}}(\omega_r)=(\alpha_1+\sum_{k=1}^m\alpha_{i_k})$ (for $B_n$) and $v_{\underline{i}}(\omega_r)=(\frac{1}{2}\alpha_1+\sum_{k=1}^m\alpha_{i_k})$ (for $C_n$). We have $v_{\underline{i}}= ([1,n]\backslash \{1, [\underline{i}], [\underline{i^{\prime}}]\}\uparrow, -\underline{i^{\prime}}, 1, \underline{i})$, where $\underline{i^{\prime}} = (i_1-1, i_2-1, \ldots, i_m-1) \in J_{m, [2, n-1]}$.
     
     (iv) Let $r=n$. Then $v=(2, 3, 4, \ldots, n-1, n, 1)$.
     
         \end{proposition}
        
   We prove this proposition after proving the following lemma.
          
   \begin{lemma} Let $v$, $v_{\underline{i}}$ are as defined in Proposition 4.1.

(i) Let $r=1$. Then $w > v$ and $l(w)=l(v)+1$ iff $w=s_kv$ where $k$ takes the following values  
$\left
\{\begin{array}{lr}
   \
    k  \in \{2, 4, 6, \ldots, n-2, n\},& {n \,\, is \,\, even}\\
    k  \in \{1, 3, 5, \ldots, n-2, n\},& {n \,\, is \,\, odd}.
    \end{array}\right.$

 (ii) Let $2 \leq r \leq n-1$. Then  $w > v_{\underline{i}}$ and  $l(w)=l(v_{\underline{i}})+1$ 
   if and only if $w$ is either $s_{\alpha_{i_t}}v_{\underline{i}}$ or 
    $s_{\alpha_{i_t}+\alpha_{{i_t}+1}} v_{\underline{i}}$ for some $i_t$ such that $|i_t-i_{t+1}| \geq 3$ or
    $s_{\alpha_{i_t}+\alpha_{{i_t}-1}} v_{\underline{i}}$ for some $i_t$ such that $|i_t-i_{t-1}| \geq 3$.

(iii) Let $r=n$. Then $w > v$ and $l(w)=l(v)+1$ iff $w=s_1v$. 
    \end{lemma}
    
     \begin{proof} We will prove this lemma for case (ii) and $(n+1)-r = 2m$. For other cases the proof is similar.
     	
     	Let $w > v_{\underline{i}} $ and $l(w)=l(v_{\underline{i}})+1$. Then $w=s_\beta v_{\underline{i}}$, for some positive root $\beta$ such that height of $\beta$ is less than or equal to $2$. So, $\beta$ is either $\alpha_j$ or $\alpha_j+\alpha_{j+1}$ for some simple root $\alpha_j$.
    
    {\bf Case 1}: $\beta =\alpha_j$.
    
     If $\beta = \alpha_{i_t}$ then $s_{\beta}v_{\underline{i}}(\omega_r)= \sum_{k\neq t}\alpha_{i_k}-\alpha_{i_t}$. Since $s_{\beta}v_{\underline{i}}(\omega_r) < v_{\underline{i}}(\omega_r)$ and $\omega_r$ is a dominant weight we have  $s_{\beta}v_{\underline{i}} >v_{\underline{i}}$. Since $\beta$ is a simple root we have $l(s_\beta v_{\underline{i}})=l(v_{\underline{i}})+1$.
    
   If $\beta = \alpha_{j}, j \neq i_t$, then $s_{\beta}v_{\underline{i}}(\omega_r) \geq v_{\underline{i}}(\omega_r)$. So $s_{\beta}v_{\underline{i}} \leq v_{\underline{i}}$ in $W^{I_r}$, a contradiction.

   
   {\bf Case 2}: $\beta = \alpha_j+\alpha_{j+1}$.
   
    If $j=i_t$ and $|i_{t+1}-i_t|\geq 3$ then  $s_{\beta}v_{\underline{i}}(\omega_r)= \sum_{k\neq t}\alpha_{i_k}-\alpha_{i_{t}+1}$. So $s_{\beta}v_{\underline{i}}(\omega_r)<v_{\underline{i}}(\omega_r)$ and hence  $s_{\beta}v_{\underline{i}}>v_{\underline{i}}$. Now we will show that $l(s_\beta v_{\underline{i}})=l(v_{\underline{i}})+1$ for this $\beta$.
  
  Note that $s_{\beta}v_{\underline{i}}=(-\underline{i},\hat{\underline{i^{\prime}}},[-n,-1]\backslash\{[-\underline{i}],[-\hat{\underline{i^{\prime}}}]\}\uparrow,[1,n]\backslash\{[\underline{i}],[\hat{\underline{i^{\prime}}}]\}\uparrow,-\hat{\underline{i^{\prime}}},\underline{i})$ where  $\hat{\underline{i^{\prime}}}=(i_1-1,i_2-1,\ldots,i_{t-1}-1,i_t+1,i_{t+1}-1\ldots,i_m-1)$. In $v_{\underline{i}}$, the position of $i_t+1$ is right to the position of $i_t-1$ and left to $i_t$ but in $s_{\beta}v_{\underline{i}}$, the position of $i_t$ remains unchanged and the positions of $i_t+1$ and $i_t-1$ are interchanged. Similarly in $v_{\underline{i}}$ the position of $-i_t-1$ is right to $-i_t$ and left to $-i_t+1$ but in $s_{\beta}v_{\underline{i}}$ the position of $-i_t$ remains unchanged and the positions of $-i_t-1$ and $-i_t+1$ are interchanged. So $inv(s_{\beta}v_{\underline{i}})=inv(v_{\underline{i}})+2$. Hence $l(s_{\beta}v_{\underline{i}})=l(v_{\underline{i}})+1$. 
   
  If $j=i_t$ and $|i_{t+1}-i_t|=2$ then since  $s_{\beta}v_{\underline{i}}(\omega_r)=v_{\underline{i}}(\omega_r)$, we have $s_{\beta}v_{\underline{i}}=v_{\underline{i}}$ in $W^{I_r}$, a contradiction.
  
  If $j, j+1 \neq i_t$ then $s_{\beta}v_{\underline{i}}(\omega_r) \geq v_{\underline{i}}(\omega_r)$. So $s_{\beta}v_{\underline{i}} \leq v_{\underline{i}}$ in $W^{I_r}$, a contradiction.
  
  {\bf Case 3.} $\beta = \alpha_{j-1}+\alpha_j$. In this case the proof is similar to the previous case.
 
  The converse part is clear from the definition of $v_{\underline{i}}$.
 \end{proof}
 
    \textbf{Proof of proposition $4.1$:}
    We prove case (ii). The proofs of other cases are similar.

We prove that for any $\underline{i}\in J_{m,[2,n]}$ there exists $v_{\underline{i}}\in W^{I_r}$ such that $v_{\underline{i}}(\omega_r)=\sum_{k=1}^m\alpha_{i_k}$.

Note that, $$\omega_r=2m(\alpha_1+\alpha_2+\cdots+\alpha_r)+
\displaystyle\sum_{i=1}^{2m-1}(2m-i)\alpha_{r+i}, 2 \leq r \leq n-1.$$

Now consider the partial order on $J_{m,[2,n]}$, given by $(i_1,i_2,\ldots,i_m)\leq(j_1,j_2,\ldots,j_m)$ if $i_k\leq j_k$, $\forall k$ and $(i_1,i_2,\ldots,i_m) < (j_1,j_2,\ldots,j_m)$ if $i_k \leq j_k$ $\forall k$ and $i_k<j_k$ for some $k$. We will prove the theorem by induction on this order.

For $(j_1,j_2,\ldots,j_m)=(n-(2m-2),n-(2m-4),\ldots,n-2,n)$, we have $v_{\underline{j}}(\omega_r)=([1,n]\backslash\{[\underline{j}],[\underline{j^{\prime}}]\},\underline{j},\underline{j^{\prime}})(\omega_r)=\displaystyle\sum_{t=1}^m\alpha_{n-2m+2t}$ where $\underline{j}=(j_1,j_2,\ldots,j_m)$ and $\underline{j^{\prime}}=(j_1-1,j_2-1,\ldots,j_m-1)$. For $(i_1,i_2,\ldots,i_m)\in J_{m,[2,n]}$ not maximal, there exists $t$ maximal such that $i_t<n-2m+2t$. Now $(i_1,i_2,\ldots,i_{t-1},i_t+1,i_{t+1},\ldots,i_m)\in J_{m,[2,n]}$ and $(i_1,i_2,\ldots,i_{t-1},i_t+1,i_{t+1},\ldots,i_m)>(i_1,i_2,\ldots,i_m)$. So by induction, there exists $w \in W^{I_r}$ such that $w(\omega_r)=\sum_{k\neq t} \alpha_{i_k}+\alpha_{1+i_t}$. Now $s_{1+i_t}s_{i_t}w(\omega_r)=\sum_{k=1}^m \alpha_{i_k}$. So for any $(i_1,i_2,\ldots,i_m)\in J_{m,[2,n]}$, there exists $v_{\underline{i}} \in W^{I_r}$ such that $v_{\underline{i}}\omega_r=\sum_{k=1}^m \alpha_{i_k}$.

For $\underline{i}\in J_{m,[2,n]}$, if there exists another $u_{\underline{i}} \in W^{I_r}$ such that $u_{\underline{i}}(\omega_r)=v_{\underline{i}}(\omega_r)$ we have $u_{\underline{i}}=v_{\underline{i}}$ in $W^{I_r}$.
This gives the uniqueness of $v_{\underline{i}}$.


Now we will prove that the $v_{\underline{i}}$'s in $W^{I_r}$ having this property are maximal.
    
   If $v_{\underline{i}}$ is not maximal, then there exists $\beta$ $\in$ $R^+$ such that $s_\beta v_{\underline{i}} > v_{\underline{i}}$ with $s_\beta v_{\underline{i}}(\omega_r) \geq 0$. We may assume that $l(s_\beta v_{\underline{i}})=l(v_{\underline{i}})+1$. So by Lemma 4.2, $\beta$ is either $\alpha_{i_t}$ or $\alpha_{i_t}+\alpha_{{i_t}+1}$ or $\alpha_{i_t}+\alpha_{{i_t}-1}$. 
    
    If $\beta = \alpha_{i_t}$, we have $s_\beta v_{\underline{i}}(\omega_r)= \sum_{k\neq t}\alpha_{i_k}-\alpha_{i_t}$ and if $\beta = \alpha_{i_t}+\alpha_{i_t+1}$, we have $s_\beta v_{\underline{i}}(\omega_r) =  \sum_{k\neq t}\alpha_{i_k}-\alpha_{i_{t+1}},$ for $|i_t-i_{t+1}| \geq 3$. Similarly, if $\beta = \alpha_{i_t}+\alpha_{i_t-1}$, we have $s_\beta v_{\underline{i}}(\omega_r) =  \sum_{k\neq t}\alpha_{i_k}-\alpha_{i_{t-1}},$ for $|i_t-i_{t-1}| \geq 3$. 
    
    So in all these cases, we see that $s_\beta v_{\underline{i}}$ is not greater than 0, a contradiction. Thus all $v_{\underline{i}}$'s are maximal having the property that $v_{\underline{i}}(\omega_r) \geq 0$.
    
   It remains to show that above listed $v$'s are the only maximal elements having the property that $v(\omega_r)\geq 0$. Let $\lambda=\sum_{t=1}^m\alpha_{i_t}$ be in the weight lattice such that $\langle\alpha_{i_k},\alpha_{i_k+1}\rangle\neq 0$ for some $k$. Let $w \in W^{I_r}$ be such that $w(\omega_r)=\lambda$. Note that $0 \leq s_{i_k}w(\omega_r)=\sum_{j\neq k} \alpha_{i_j}<\lambda$. Hence $s_{i_k}w>w$. This implies that $w$ is not maximal having the property that $w(\omega_r) \geq 0$.
    
   
  \begin{proposition}
  The set of all minimal $w$ in $W^{I_r}$ such that $w(\omega_r) \leq 0 $ for type $B_n$ and $C_n$ are the following:  
   \end{proposition}
   (i) Let $r=1$. Then
   \[
    w=\left\{\begin{array}{lr}
    (-n, -(n-2), \ldots, -4, -2, 1, 3, \ldots, n-3, n-1),& \text{if n is even}\\
     (-n, -(n-2), \ldots, -3, -1, 2, 4, \ldots, n-3, n-1),& \text{if n is odd.}\\
    \end{array}\right.
    \]
   
   (ii) Let $2 \leq r \leq n-1$. For $(n+1)-r=2m$, $\underline{i} \in J_{m,[2, n]}$ and $v_{\underline{i}}(\omega_r) = \sum_{k=1}^m \alpha_{i_k}$, we have $w_{\underline{i}}=s_{i_1}s_{i_2}\ldots s_{i_m}v_{\underline{i}}=([1,n]\backslash \{[\underline{i}], [\underline{i^{\prime}}]\}\uparrow, -\underline{i}, \underline{i^{\prime}})$ where $\underline{i^{\prime}}=(i_1-1, i_2-1,\ldots,i_m-1)$.
   In this case $w_{\underline{i}}(\omega_r)=-v_{\underline{i}}(\omega_r)$.
   
   (iii) Let $2 \leq r \leq n-1$. For $(n+1)-r=2m+1$, $\underline{i} \in J_{m,[3, n]}$ and $v_{\underline{i}}(\omega_r) = \alpha_1+ \sum_{k=1}^m \alpha_{i_k}$, we have $w_{\underline{i}}=s_1s_{i_1}s_{i_2}\ldots s_{i_m}v_{\underline{i}}= ([1,n]\backslash \{1, [\underline{i}], [\underline{i^{\prime}}]\}\uparrow, -\underline{i}, -1, \underline{i^{\prime}})$, where $\underline{i^{\prime}} = (i_1-1, i_2-1, \ldots, i_m-1) \in J_{m, [2, n-1]}$. In this case also $w_{\underline{i}}(\omega_r)=-v_{\underline{i}}(\omega_r)$. 
   
   (iv) For $r= n$, $w = (2, 3, \ldots, n-1, n, -1)$.
   \begin{proof} 
   For the proof of minimality of $w$ refer to \cite{KP}.
    \end{proof}
  
  \begin{proposition}
  Let $v$, $w$, $v_{\underline{i}}$, and $w_{\underline{j}}$ be as stated in Proposition $4.1$ and Proposition $4.3$ respectively.
 
(i) For $2 \leq r \leq n-1$, $X^{v_{\underline{i}}}_{w_{\underline{j}}}$ is nonempty if and only if $|i_k-j_k| \leq 1$ $\forall 1 \leq k \leq m$.

(ii) For $r=1, n$, $X^{v}_{w}$ is non-empty for any $v$ and $w$.
  \end{proposition}
    
  \begin{proof} We will prove this lemma for $2 \leq r \leq n-1$ and $(n+1)-r=2m$. For other cases the proof is similar.
  
  Let $X^{v_{\underline{i}}}_{w_{\underline{j}}}$ be nonempty. So $v_{\underline{i}} < w_{\underline{j}}$. Since $l(v_{\underline{i}}) = l(v_{\underline{j}})$ we have $l(w_{\underline{j}})-l(v_{\underline{i}})=m$. By the repeated application of Lemma 4.2 we have $w_{\underline{j}} = \displaystyle \prod_{ \mathbf{card}\{\beta\} = m}s_\beta v_{\underline{i}}$, for some $\beta$ such that $\beta$'s is either $\alpha_{i_t}$ or $\alpha_{i_t}+\alpha_{{i_t}+1}$ or $\alpha_{i_t}+\alpha_{{i_t}-1}$ for some $i_t$. So $w_{\underline{j}}(\omega_r) = -\displaystyle \sum_{{j_k}:|i_k-j_k| \leq 1} \alpha_{j_k}$. Hence, $|i_k-j_k| \leq 1$ $\forall 1 \leq k \leq m$.
  
  Conversely, let $|i_k-j_k| \leq 1$ for all $1 \leq k \leq m$. So, $w_{\underline{j}}=\displaystyle \prod_{t\in \{t:i_t=j_t\}}s_{\alpha_{i_t}}$ $ \displaystyle \prod_{t\in\{t:j_t=i_t-1\}\downarrow}s_{\alpha_{i_t}+\alpha_{i_t-1}}$ $ \displaystyle \prod_{t\in\{t:j_t=i_t+1\}\uparrow}s_{\alpha_{i_t}+\alpha_{i_t+1}}$ $v_{\underline{i}}$. The arrows $\uparrow$ and $\downarrow$ denote that the reflections in the product are applied in increasing and decreasing order of $t$ respectively. Since $\underline{i}, \underline{j} \in J_{m,[2,n]}$ and $|i_k-j_k|\leq 1$, we see that the sets $\{\bigcup\limits_{\{t:j_t=i_t+1\}}\{i_t-1,i_t,i_t+1\}\}$, $\{\bigcup\limits_{\{t:j_t=i_t-1\}}\{i_t-2,i_t-1,i_t\}\}$ and $\{\bigcup\limits_{\{t:j_t=i_t\}}\{i_t-1,i_t\}\}$ are mutually disjoint.
  
  
  In the first step we see that when we multiply $v_{\underline{i}}$ by $s_{\alpha_{i_t}+\alpha_{i_t+1}}$ in increasing order of $t \in \{t: j_t=i_t+1\}$ then after each multiplication the product is greater than $v_{\underline{i}}$ and the length increases by $1$. Let $t$ be maximal such that $j_t=i_t+1$. By lemma $4.2$,  $s_{\alpha_{i_t}+\alpha_{i_t+1}}v_{\underline{i}}>
  v_{\underline{i}}$ and $l(s_{\alpha_{i_t}+\alpha_{i_t+1}}v_{\underline{i}})=
  l(v_{\underline{i}})+1$. 
  
  Let $t$ be such that $i_{t-1}=i_t-2$ and $j_{t-1}=i_{t-1}+1$. Note that $s_{\alpha_{i_t}+\alpha_{i_t+1}}v_{\underline{i}}=(-\underline{i},\hat{\underline{i^{\prime}}},[-n,-1]\backslash\{[-\underline{i}],[-\hat{\underline{i^{\prime}}}]\}\uparrow,[1,n]\backslash\{[\underline{i}],[\hat{\underline{i^{\prime}}}]\}\uparrow,-\hat{\underline{i^{\prime}}},\underline{i})$, where  $\hat{\underline{i^{\prime}}}=(i_1-1,i_2-1,\ldots,i_t-3,i_t+1,i_{t+1}-1\ldots,i_m-1)$ and $s_{\alpha_{i_t-2}+\alpha_{i_t-1}}s_{\alpha_{i_t}+\alpha_{i_t+1}}
  v_{\underline{i}} = (-\underline{i},\hat{\hat{\underline{i}^{\prime}}},[-n,-1]\backslash\{[-\underline{i}],[-\hat{\hat{\underline{i}^{\prime}}}]\}\uparrow,[1,n]\backslash\{[\underline{i}],[\hat{\hat{\underline{i}^{\prime}}}]\}\uparrow,-\hat{\hat{\underline{i}^{\prime}}},\underline{i})$, where $\hat{\hat{\underline{i}^{\prime}}}=(i_1-1,i_2-1,\ldots,i_t-1,i_t+1,i_{t+1}-1\ldots,i_m-1)$. In $s_{\alpha_{i_t}+\alpha_{i_t+1}}v_{\underline{i}}$, the position of $i_t-3$ is left to the positions of both $i_t-1$ and $i_t-2$ but in $s_{\alpha_{i_t-2}+\alpha_{i_t-1}}s_{\alpha_{i_t}+\alpha_{i_t+1}}
  v_{\underline{i}}$, the position of $i_t-2$ remains unchanged and the positions of $i_t-1$ and $i_t-3$ are interchanged. Similarly the positions of $-i_t+1$ and $-i_t+3$ are also interchanged in $s_{\alpha_{i_t-2}+\alpha_{i_t-1}}s_{\alpha_{i_t}+\alpha_{i_t+1}}
  v_{\underline{i}}$. So $inv(s_{\alpha_{i_t-2}+\alpha_{i_t-1}}
  s_{\alpha_{i_t}+\alpha_{i_t+1}}
  v_{\underline{i}})=inv(s_{\alpha_{i_t}+\alpha_{i_t+1}}
  v_{\underline{i}})+2$. Hence $l(s_{\alpha_{i_t-2}+\alpha_{i_t-1}}
  s_{\alpha_{i_t}+\alpha_{i_t+1}}
  v_{\underline{i}})=l(s_{\alpha_{i_t}+\alpha_{i_t+1}}v_{\underline{i}})+1$. By lemma $2.3$ we see that $s_{\alpha_{i_t}+\alpha_{i_t+1}}
  v_{\underline{i}}<s_{\alpha_{i_t-2}+\alpha_{i_t-1}}s_{\alpha_{i_t}+\alpha_{i_t+1}}
  v_{\underline{i}}$. 
  
   Repeating this process we can see that $\displaystyle\prod_{t\in\{t:j_t=i_t+1\}\uparrow}s_{\alpha_{i_t}+\alpha_{i_t+1}} v_{\underline{i}} > v_{\underline{i}}$ and the length is increased by the number of reflections multiplied.
  

   Since $\{\bigcup\limits_{\{t:j_t=i_t+1\}}\{i_t-1,i_t,i_t+1\}\}, \{\bigcup\limits_{\{t:j_t=i_t-1\}}\{i_t-2,i_t-1,i_t\}\}$ and $\{\bigcup\limits_{\{t:j_t=i_t\}}\{i_t-1,i_t\}\}$ are mutually disjoint, by repeating the above process we see that $\displaystyle\prod_{\{{t:j_t=i_t}\}}s_{\alpha_{i_t}} \prod_{t\in\{t:j_t=i_t-1\}\downarrow}s_{\alpha_{i_t}+\alpha_{i_t-1}} \\ \prod_{t\in\{t:j_t=i_t+1\}\uparrow} s_{\alpha_{i_t}+\alpha_{i_t+1}} v_{\underline{i}}>v_{\underline{i}}$ and the length is increased by the number of reflections multiplied. So $w_{\underline{j}} > v_{\underline{i}}$ and hence $X^{v_{\underline{i}}}_{w_{\underline{j}}}$ is nonempty.
  \end{proof}

{\bf Remark:} Note that $\underline{i}$ denotes the positions of the simple roots with nonzero coefficients in $v_{\underline{i}}(\omega_r)$ and similarly, $\underline{j}$ denotes the positions of the simple roots with nonzero coefficients in $w_{\underline{j}}(\omega_r)$. 
 
  \begin{theorem}
   Let $v$, $w$, $v_{\underline{i}}$ and $w_{\underline{j}}$ be as stated in Proposition $4.1$ and Proposition $4.3$ respectively.
 
(i) For $2 \leq r \leq n-1$, $(X{^{v_{\underline{i}}}_{w_{\underline{j}}}})^{ss}_T(\mathcal{L}_r)$ is nonempty if and only if $\underline{i} = \underline{j}$.

(ii) For $r =1$ and $n$, $(X^v_w)^{ss}_T$ is non-empty for any $v$ and $w$.
  
  \end{theorem}
  
  \begin{proof} We will prove this theorem for case (i). Proof of case (ii) is similar. Let $\underline{i} = \underline{j}$. Then we have $v_{\underline{i}}(\omega_r)$+$w_{\underline{j}}(\omega_r)=0$. So $p_{v_{\underline{i}}}p_{w_{\underline{j}}}$ is a non-zero $T$-invariant section of $\mathcal{L}_r$ on $G/P_r$ which does not vanish identically on $X^{v_{\underline{i}}}_{w_{\underline{j}}}$. Hence, $(X^{v_{\underline{i}}}_{w_{\underline{j}}})^{ss}_T (\mathcal{L}_r)$ is non-empty.

Conversely, if $\underline{i} \neq \underline{j}$, then there exists $t$ such that $j_t \neq i_t$. Since $X^{v_{\underline{i}}}_{w_{\underline{j}}} \neq \emptyset$, by Proposition $4.4$ we have $j_t = i_t+1$ or $j_t = i_t-1$. If $j_t = i_t+1$ then $w_{\underline{j}}(\omega_r)$ = $-\sum_{k \neq t}{\alpha_{i_k}}-\alpha_{i_t+1}$ and if $j_t = i_t-1$ then $w_{\underline{j}}(\omega_r)$ = $-\sum_{k \neq t}{\alpha_{i_k}}-\alpha_{i_t-1}$. Let $u \in W^{I_r}$ be such that $v_{\underline{i}} \leq u \leq w_{\underline{j}}$. Then $u$ is of the form $u = (\displaystyle \prod_{\beta}s_\beta) v_{\underline{i}}$, where $\beta$'s are some positive roots. For $j_t=i_t+1$ at most one $\beta$ can be $\alpha_{i_t}+\alpha_{i_t+1}$ and none of the other $\beta$'s can contain $\alpha_{i_t}$ or $\alpha_{i_t+1}$ as a summand. So in $u(\omega_r)$, the coefficient of $\alpha_{i_t}$ is either zero or one and the coefficient of $\alpha_{i_t+1}$ is either zero or $-1$. Similarly for $j_t=i_t-1$ at most one $\beta$ can be $\alpha_{i_t}+\alpha_{i_t-1}$ and none of the other $\beta$'s can contain $\alpha_{i_t}$ or $\alpha_{i_t-1}$ as a summand. So in $u(\omega_r)$ the coefficient of $\alpha_{i_t}$ is either zero or one and the coefficient of $\alpha_{i_t-1}$ is either zero or $-1$. For $j_t=i_t+1$, $u(\omega_r)$ contains either $\alpha_{i_t}$ or $\alpha_{i_t+1}$ as a summand and for $j_t=i_t-1$, $u(\omega_r)$ contains either $\alpha_{i_t}$ or $\alpha_{i_t-1}$ as a summand. Hence there  does not exist a sequence $v_{\underline{i}} = u_1 \leq u_2 \leq \ldots \leq u_k = w_{\underline{j}}$ such that $\sum_{l=1}^ku_l(\omega_r)=0$ and so we don't have a non-zero $T$- invariant section of $\mathcal L_r$ which is not identically zero on $X^{v_{\underline{i}}}_{w_{\underline{j}}}$. So, we conclude that the set $(X_{w_{\underline{j}}}^{v_{\underline{i}}})^{ss}_T(\mathcal{L}_r)$ is empty.  
 \end{proof}

We illustrate Proposition $4.4$ and Theorem $4.5$ with an example.  
 
{\bf Example:} $B_5$, $\omega_4 = (2,2,2,2,1)$
 
 $\underline{i}$ \hspace{2.7cm} $v_{\underline{i}}$ \hspace{2.2cm} $v_{\underline{i}}(\omega_2)$ \hspace{2cm} $w_{\underline{i}}(\omega_2)$ \hspace{2cm} $w_{\underline{i}}$
 
 (2) \hspace{1.5cm} $(3,4,5,-1,2)$ \hspace{.5cm} $(0,1,0,0,0)$ \hspace{1cm} $(0,-1,0,0,0)$ \hspace{.5cm} $(3,4,5,-2,1)$

(3) \hspace{1.5cm} $(1,4,5,-2,3)$ \hspace{.5cm} $(0,0,1,0,0)$ \hspace{1cm} $(0,0,-1,0,0)$ \hspace{.5cm} $(1,4,5,-3,2)$

(4) \hspace{1.5cm} $(1,2,5,-3,4)$ \hspace{.5cm} $(0,0,0,1,0)$ \hspace{1cm} $(0,0,0,-1,0)$ \hspace{.5cm} $(1,2,5,-4,3)$

(5) \hspace{1.5cm} $(1,2,3,-4,5)$ \hspace{.5cm} $(0,0,0,0,1)$ \hspace{1cm} $(0,0,0,0,-1)$ \hspace{.5cm} $(1,2,3,-5,4)$

 So from the above observation, $X_{w_{(2)}}^{v_{(2)}}$, $X_{w_{(2)}}^{v_{(3)}}$, $X_{w_{(3)}}^{v_{(2)}}$, $X_{w_{(3)}}^{v_{(3)}}$, $X_{w_{(3)}}^{v_{(4)}}$, $X_{w_{(4)}}^{v_{(3)}}$ and $X_{w_{(4)}}^{v_{(4)}}$, $X_{w_{(4)}}^{v_{(5)}}$, $X_{w_{(5)}}^{v_{(4)}}$, $X_{w_{(5)}}^{v_{(5)}}$ are all non-empty. We have $(X_{w_{(2)}}^{v_{(3)}})^{ss}_T(\mathcal{L}_4)$, $(X_{w_{(3)}}^{v_{(2)}})^{ss}_T(\mathcal{L}_4)$,
 $(X_{w_{(3)}}^{v_{(4)}})^{ss}_T(\mathcal{L}_4)$,
 $(X_{w_{(4)}}^{v_{(3)}})^{ss}_T(\mathcal{L}_4)$, $(X_{w_{(4)}}^{v_{(5)}})^{ss}_T(\mathcal{L}_4)$ and $(X_{w_{(5)}}^{v_{(4)}})^{ss}_T(\mathcal{L}_4)$ are empty whereas  $(X_{w_{(2)}}^{v_{(2)}})^{ss}_T(\mathcal{L}_4)$, $(X_{w_{(3)}}^{v_{(3)}})^{ss}_T(\mathcal{L}_4)$,
 $(X_{w_{(4)}}^{v_{(4)}})^{ss}_T(\mathcal{L}_4)$ and 
 $(X_{w_{(5)}}^{v_{(5)}})^{ss}_T(\mathcal{L}_4)$ are non-empty.   

\section{Type D}

As in types $B$ and $C$ here also for a fundamental weight $\omega_r$ we describe all the maximal $v \in W^{I_r}$ and minimal $w \in W^{I_r}$ such that $v(\omega_r) \geq 0$ and $w (\omega_r) \leq 0$ and then we use the same techniques to describe $v,w \in W^{I_r}$ for which  the Richardson variety $X_w^v$ have nonempty semistable locus for the action of a maximal torus $T$ and with respect to the line bundle $\mathcal L_r$.  
  
\begin{proposition}
 
     Let $G$ be of type $D_n$. Let $v \in W^{I_r}$ be maximal such that $v(\omega_r) \geq 0$. Then the description of $v$ is the following:  
     
      (i) For $r=1$, \,\, $v(4\omega_{1})=\left\{\begin{array}{lr}
      \
      2\alpha_2+2\sum_{i=2}^{\frac{n}{2}}\alpha_{2i},& n\equiv 0(mod \,\, 4) \\
      2\alpha_1+2\sum_{i=2}^{\frac{n}{2}}\alpha_{2i},& n\equiv 2(mod \,\, 4) \\
      \alpha_1+3\alpha_2+2\alpha_3+2\sum_{i=2}^{\frac{n-1}{2}}\alpha_{2i+1},& n\equiv 1(mod \,\, 4) \\
      3\alpha_1+\alpha_2+2\alpha_3+2\sum_{i=2}^{\frac{n-1}{2}}\alpha_{2i+1},& n\equiv 3(mod \,\, 4). 
      \end{array}\right.$
      
       where $v=\left\{\begin{array}{lr}
       \
       (-(n-1),-(n-3),\ldots,-3,-1,2,4,6,\ldots,n-2,n),& n\equiv 0(mod \,\, 4) \\
       (-(n-1),-(n-3),\ldots,-3,1,2,4,6,\ldots,n-2,n),& n\equiv 2(mod \,\, 4) \\
       (-(n-1),-(n-3),\ldots,-4,-1,2,3,5,\ldots,n-2,n),& n\equiv 1(mod \,\, 4) \\
       (-(n-1),-(n-3),\ldots,-4,1,2,3,5,\ldots,n-2,n),& n\equiv 3(mod \,\, 4). 
       \end{array}\right.$
       
       (ii) For $r=2$, $v(4\omega_{2})=\left\{\begin{array}{lr}
       \
       2\alpha_1+2\sum_{i=2}^{\frac{n}{2}}\alpha_{2i},& n\equiv 0(mod \,\, 4) \\
       2\alpha_2+2\sum_{i=2}^{\frac{n}{2}}\alpha_{2i},& n\equiv 2(mod \,\, 4) \\
       3\alpha_1+\alpha_2+2\alpha_3+2\sum_{i=2}^{\frac{n-1}{2}}\alpha_{2i+1},& n\equiv 1(mod \,\, 4) \\
       \alpha_1+3\alpha_2+2\alpha_3+2\sum_{i=2}^{\frac{n-1}{2}}\alpha_{2i+1},& n\equiv 3(mod \,\, 4),
       \end{array}\right.$ and  \\
       
        where $v=\left\{\begin{array}{lr}
       \
       ((n-1),-(n-3),\ldots,-3,1,2,4,6,\ldots,n-2,n),& n\equiv 0(mod \,\, 4) \\
       ((n-1),-(n-3),\ldots,-3,-1,2,4,6,\ldots,n-2,n),& n\equiv 2(mod \,\, 4) \\
       ((n-1),-(n-3),\ldots,-4,1,2,3,5,\ldots,n-2,n),& n\equiv 1(mod \,\, 4) \\
       ((n-1),-(n-3),\ldots,-4,-1,2,3,5,\ldots,n-2,n),& n\equiv 3(mod \,\, 4). 
       \end{array}\right.$
       
      (iii) Let $3 \leq r \leq n-1$. For $(n+1)-r = 2m$ and for any $\underline{i} = (i_1,i_2,\ldots,i_m) \in J_{m,[1,n]} \backslash Z$, there exists an unique $v_{\underline{i}} \in W^{I_r}$ such that $v_{\underline{i}}(\omega_r) = \displaystyle\sum_{k=1}^m \alpha_{i_k}$, where $Z = \{(1,3,i_1,i_2,\ldots,i_{m-2}):i_k \in \{5,\ldots,n-1,n\}$ and $i_{k+1}-i_k \geq 2, \forall k\}$. 
      
     (a) For $\underline{i} \in J_{m,[3,n]}$, $v_{\underline{i}}(\omega_r) = \displaystyle\sum_{k=1}^m \alpha_{i_k}$, $v_{\underline{i}}=\left\{\begin{array}{lr}
     \
     (-1,[2,n]\backslash\{[\underline{i}],[\underline{i^{\prime}}]\}\uparrow,-\underline{i^{\prime}},\underline{i}),& \text{m is odd} \\
     (1,[2,n]\backslash\{[\underline{i}],[\underline{i^{\prime}}]\}\uparrow,-\underline{i^{\prime}},\underline{i}),& \text{m is even}. 
          \end{array}\right.$  
   
        (b) For $\underline{i} \in J_{m-1,[4,n]}$, $v_{1,\underline{i}}(\omega_r) = \alpha_1+\displaystyle\sum_{k=1}^{m-1} \alpha_{i_k}$, with \\
        $v_{1,\underline{i}}=\left\{\begin{array}{lr}
     \
     ([3,n]\backslash\{[\underline{i}],[\underline{i^{\prime}}]\}\uparrow,-\underline{i^{\prime}},1,2,\underline{i}),& \text{m is odd} \\
     (-t,[3,n]\backslash\{t,[\underline{i}],[\underline{i^{\prime}}]\}\uparrow,-\underline{i^{\prime}},1,2, \underline{i}),& \text{m is even}. 
          \end{array}\right.$
         where $t=min\{[3,n]\backslash\{[\underline{i}],[\underline{i^{\prime}}]\}\}$.  
            
    
     (c) For $\underline{i} \in J_{m-1,[4,n]}$, $v_{2,\underline{i}}(\omega_r) = \alpha_2+\displaystyle\sum_{k=1}^{m-1} \alpha_{i_k}$ with \\ 
     $v_{2,\underline{i}}=\left\{\begin{array}{lr}
     \
     (-t,[3,n]\backslash\{t,[\underline{i}],[\underline{i^{\prime}}]\}\uparrow,-\underline{i^{\prime}},-1,2,\underline{i}),& \text{m is odd} \\
     ([3,n]\backslash\{[\underline{i}],[\underline{i^{\prime}}]\}\uparrow,-\underline{i^{\prime}},-1,2, \underline{i}),& \text{m is even}, 
     \end{array}\right.$ where $t=min\{[3,n]\backslash\{[\underline{i}],[\underline{i^{\prime}}]\}\}.$
     
    (iv) Let $3 \leq r \leq n-1$. For $(n+1)-r=2m+1$ and for any $\underline{i}=(i_1,i_2,\ldots,i_m) \in J_{m,[4,n]}$, there exists unique $v_{\underline{i}} \in W^{I_r}$ such that $v_{\underline{i}}(\omega_r) =\frac{1}{2}\alpha_{1}+\frac{1}{2}\alpha_{2}+\displaystyle\sum_{k=1}^{m} \alpha_{i_k}$. Also, for any $\underline{i}=(i_1,i_2,\ldots,i_{m-1}) \in J_{m-1,[5,n]}$, there exists unique $v_{\underline{i},1} \in W^{I_r}$ such that $v_{\underline{i},1}(\omega_r)=\frac{3}{2}\alpha_{1}+\frac{1}{2}\alpha_{2}+\alpha_{3}+\displaystyle\sum_{k=1}^{m-1} \alpha_{i_k}$ and there exists unique $v_{\underline{i},2} \in W^{I_r}$ such that $v_{\underline{i},2}(\omega_r)=\frac{1}{2}\alpha_1+\frac{3}{2}\alpha_{2}+\alpha_{3}+\displaystyle\sum_{k=1}^{m-1} \alpha_{i_k}$. We have:
     
 (a)  $v_{\underline{i}}=\left\{\begin{array}{lr}
	\
	(-1,[3,n]\backslash\{[\underline{i}],[\underline{i^{\prime}}]\}\uparrow,-\underline{i^{\prime}},2,\underline{i}),& \text{m is odd} \\
	(1,[3,n]\backslash\{[\underline{i}],[\underline{i^{\prime}}]\}\uparrow,-\underline{i^{\prime}},2,\underline{i}),& \text{m is even}. 
     \end{array}\right.$

(b) $v_{\underline{i},1}=\left\{\begin{array}{lr}
\
(4,[5,n]\backslash\{[\underline{i}],[\underline{i^{\prime}}]\}\uparrow,-\underline{i^{\prime}},1,2,3,\underline{i}),& \text{m is odd} \\
(-4,[5,n]\backslash\{[\underline{i}],[\underline{i^{\prime}}]\}\uparrow,-\underline{i^{\prime}},1,2,3,\underline{i}),& \text{m is even}. 
\end{array}\right.$
        
 (c) $v_{\underline{i},2}=\left\{\begin{array}{lr}
 \
 (-4,[5,n]\backslash\{[\underline{i}],[\underline{i^{\prime}}]\}\uparrow,-\underline{i^{\prime}},-1,2,3,\underline{i}),& \text{m is odd} \\
 (4,[5,n]\backslash\{[\underline{i}],[\underline{i^{\prime}}]\}\uparrow,-\underline{i^{\prime}},-1,2,3,\underline{i}),& \text{m is even}. 
 \end{array}\right.$

 (v) For $r=n$, $v=(1,3,4,5,\ldots,n-1,n,2)$ with $v(\omega_n)=\frac{1}{2}\alpha_1+\frac{1}{2}\alpha_2$.

      \end{proposition}

\begin{proof}
The proof of the proposition is similar to the proofs for type B and C which uses the following crucial lemma. 
\end{proof}
      
  \begin{lemma} Let $v,  v_{\underline{i}}, v_{\underline{i},1}, v_{\underline{i},2}, v_{1,\underline{i}}, v_{2,\underline{i}}$ are as defined in Proposition 5.1. 
 
 \noindent  (i) Let $r=1$. Then $w > v$ and $l(w) = l(v)+1$ iff $ w=s_kv$ where $k$ takes the following values: $\left\{\begin{array}{lr}
   \
     k  \in \{2, 4, 6, \ldots, n-2, n\},& {n \equiv 0(mod \,\, 4)}\\
     k  \in \{1, 4, 6, \ldots, n-2, n\},& {n \equiv 2(mod \,\, 4)}\\
     k  \in \{2, 5, 7, \ldots, n-2, n\},& {n \equiv 1(mod \,\, 4)}\\
     k  \in \{1, 5, 7, \ldots, n-2, n\},& {n \equiv 3(mod \,\, 4)}.
    \end{array}\right.$
    
    \noindent    (v)  Let $r=2$. Then $w > v$ and $l(w) = l(v)+1$ iff $ w=s_kv$ where $k$ takes the following values: $\left\{\begin{array}{lr}
   \
    k  \in \{1, 4, 6, \ldots, n-2, n\},& {n \equiv 0(mod \,\, 4)}\\
    k  \in \{2, 4, 6, \ldots, n-2, n\},& {n \equiv 2(mod \,\, 4)}\\
    k  \in \{1, 5, 7, \ldots, n-2, n\},& {n \equiv 1(mod \,\, 4)}\\
    k  \in \{2, 5, 7, \ldots, n-2, n\},& {n \equiv 3(mod \,\, 4)}.
    \end{array}\right.$
    
    \noindent (iii) Let $3 \leq r \leq n-1$ and $(n+1)-r = 2m$. Then,
  
\noindent   (a) $w > v_{\underline{i}}$ and $l(w) = l(v_{\underline{i}})+1$ iff $w = s_{\alpha_{i_t}} v_{\underline{i}}$ or $s_{\alpha_{i_t}+\alpha_{i_t+1}} v_{\underline{i}}$ for $i_t :|i_t-i_{t+1}| \geq 3$ or $s_{\alpha_{i_t}+\alpha_{i_t-1}} v_{\underline{i}}$ for $|i_t-i_{t-1}| \geq 3$ or $s_{\alpha_{1}+\alpha_{3}}$ or $s_{{\alpha_2}+{\alpha_3}}$ for $i_1 = 3$.
   
\noindent   (b) $w > v_{1,\underline{i}}$ and $l(w) = l(v_{1,\underline{i}})+1$ iff $w = s_{\alpha_{k}} v_{1,\underline{i}}$ for $k \in \{1, i_t\}$ or $s_{\alpha_{i_t}+\alpha_{i_t+1}} v_{1,\underline{i}}$ for $i_t :|i_t-i_{t+1}| \geq 3$ or $s_{\alpha_{i_t}+\alpha_{i_t-1}} v_{1,\underline{i}}$ for $i_t:|i_t-i_{t-1}| \geq 3$ with $i_1 \geq 5$ or $s_{\alpha_{1}+\alpha_{3}}$ with $i_{1} \geq 5$. 
   
\noindent   (c) $w > v_{2,\underline{i}}$ and $l(w) = l(v_{2,\underline{i}})+1$ iff $w = s_{\alpha_{k}} v_{2,\underline{i}}$ for $k \in \{2, i_t\}$ or $s_{\alpha_{i_t}+\alpha_{i_t+1}} v_{2,\underline{i}}$ for $i_t :|i_t-i_{t+1}| \geq 3$ or $s_{\alpha_{i_t}+\alpha_{i_t-1}} v_{2,\underline{i}}$ for $i_t:|i_t-i_{t-1}| \geq 3$ with $i_1 \geq 5$ or $s_{\alpha_{2}+\alpha_{3}}$ with $i_{1} \geq 5$.
      
  \noindent (iv) Let $3 \leq r \leq n-1$ and $(n+1)-r = 2m+1$. Then, 

\noindent (a) $w > v_{\underline{i}}$ and $l(w) = l(v_{\underline{i}})+1$ iff $w=  s_{\alpha_{k}} v_{\underline{i}}$ for $k \in \{1, 2, i_t\}$ 
      or $s_{\alpha_{i_t}+\alpha_{i_t+1}} v_{\underline{i}}$ for $i_t:|i_t-i_{t+1}| \geq 3$ or $s_{\alpha_{i_t}+\alpha_{i_t-1}} v_{\underline{i}}$ for $i_t:|i_t-i_{t-1}| \geq 3$ with $i_1 \geq 5$. 

\noindent (b) $w > v_{\underline{i},1}$ and $l(w) = l(v_{\underline{i},1})+1$ iff $w = s_{\alpha_{k}} v_{\underline{i},1}$ for $k \in \{1, i_t\}$ or $s_{\alpha_{i_t}+\alpha_{i_t+1}} v_{\underline{i},1}$ for $i_t:|i_t-i_{t+1}| \geq 3$ or $s_{\alpha_{i_t}+\alpha_{i_t-1}} v_{\underline{i},1}$ for $i_t:|i_t-i_{t-1}| \geq 3$ with $i_1 \geq 6$ or $s_{\alpha_1+\alpha_3+\alpha_4} v_{\underline{i},1}$ with $i_1 \geq 6$. 
      
\noindent (c) $w > v_{\underline{i},2}$ and $l(w) = l(v_{\underline{i},2})+1$ iff $w = s_{\alpha_{k}} v_{\underline{i},2}$ for $k \in \{2, i_t\}$ or $s_{\alpha_{i_t}+\alpha_{i_t+1}} v_{\underline{i},2}$ for $i_t:|i_t-i_{t+1}| \geq 3$ or $s_{\alpha_{i_t}+\alpha_{i_t-1}} v_{\underline{i},2}$ for $i_t:|i_t-i_{t-1}| \geq 3$ with $i_1 \geq 6$ or       $s_{\alpha_2+\alpha_3+\alpha_4} v_{\underline{i},2}$ for $i_1 \geq 6$. 
      
      \noindent (v) Let $r=n$. Then $w > v$ and $l(w) = l(v)+1$ iff $w = s_1v$ or $s_2v$. 
    
    \end{lemma}
     \begin{proof} We prove this lemma only for case (iv) and part (b). Proofs of the other cases are similar to this case.
  
  Let  $w \in W^{I_r}$ such that $w > v_{\underline{i},1}$ and $l(w) = l(v_{\underline{i},1}) +1$. Then $w = s_\beta v_{\underline{i},1}$, with $ht(\beta) \leq 2$ or $\beta$ = $\alpha_1+\alpha_3+\alpha_4$ when $i_1 \geq 6$. Note that $v_{\underline{i},1}(\omega_r)=\frac{3}{2}\alpha_1+\frac{1}{2}\alpha_2+\alpha_3+\displaystyle\sum_{k=1}^{m-1} \alpha_{i_k}$.

  {\bf Case 1.} $\beta = \alpha_k$, a simple root.
  
  If $k=i_t$, then $s_{\alpha_{i_t}}v_{\underline{i},1} (\omega_r)= \frac{3}{2}\alpha_1+\frac{1}{2}\alpha_2+\alpha_3+\sum_{k \neq t}\alpha_{i_k}-\alpha_{i_t}$. Since $s_\beta v_{\underline{i},1}(\omega_r) < v_{\underline{i},1}(\omega_r)$ and $\omega_r$ is a dominant weight we have $s_\beta v_{\underline{i},1} > v_{\underline{i},1}$. Since $\beta$ is a simple root so $l(s_{\beta}v_{\underline{i},1})=l(v_{\underline{i},1})+1$.
  
 If $\beta = \alpha_1$, then $s_1 v_{\underline{i},1} (\omega_r)= -\frac{1}{2}\alpha_1+\frac{1}{2}\alpha_2+\alpha_3+ \sum{\alpha_{i_k}}$. So $s_1 v_{\underline{i},1}(\omega_r) < v_{\underline{i},1}(\omega_r)$. Hence $s_1 v_{\underline{i},1} > v_{\underline{i},1}$ with $l(s_{\beta}v_{\underline{i},1})=l(v_{\underline{i},1})+1$.
 
 If $\beta = \alpha_2$ or $\alpha_3$, then $s_\beta v_{\underline{i},1} = v_{\underline{i},1}$ in $W^{I_r}$, a contradiction. 
   
   If $k \notin \{1, 2, 3, i_t\}$, then  $s_{\beta}v_{\underline{i},1}(\omega_r) > v_{\underline{i},1}(\omega_r)$. Hence $s_\beta v_{\underline{i},1} < v_{\underline{i},1}$, a contradiction.

  {\bf Case 2.} $\beta = \alpha_k+\alpha_{k+1}$, a positive root of height $2$.
  
  If $k=i_t$ with $|i_t-i_{t+1}|\geq 3$, then $s_{\beta} v_{\underline{i},1}(\omega_r) = \frac{3}{2}\alpha_1+\frac{1}{2}\alpha_2+\alpha_3+\displaystyle\sum_{k \neq t} \alpha_{i_k}-\alpha_{i_t+1} <  v_{\underline{i},1}(\omega_r)$. So $s_{\beta} v_{\underline{i},1} > v_{\underline{i},1}$. Now we prove that $l(s_{\beta}v_{\underline{i},1})=l(v_{\underline{i},1})+1$. 
  
  We will prove this case for $m$ is odd. Note that $s_{\beta}v_{\underline{i},1}=(-\underline{i},-3,-2,-1, \hat{\underline{i^{\prime}}},[-n,-5] \\ \backslash\{[-\underline{i}],[-\hat{\underline{i^{\prime}}}]\}\uparrow,-4,4,[5,n]\backslash\{[\underline{i}],[\hat{\underline{i^{\prime}}}]\}\uparrow,
  -\hat{\underline{i^{\prime}}},1,2,3,\underline{i})$ where  $\hat{\underline{i^{\prime}}}=(i_1-1,i_2-1,\ldots,i_{t-1}-1,i_t+1,i_{t+1}-1\ldots,i_m-1)$. In $v_{\underline{i},1}$, the position of $i_t+1$ is right to the position of $i_t-1$ and left to the position of $i_t$ but in $s_{\beta}v_{\underline{i},1}$, the position of $i_t$ remains unchanged and the positions of $i_t+1$ and $i_t-1$ are interchanged. Similarly in $v_{\underline{i},1}$ the position of $-i_t-1$ is right to the position of $-i_t$ and left to the position of $-i_t+1$ but in $s_{\beta}v_{\underline{i},1}$ the position of $-i_t$ remains unchanged and the positions of $-i_t-1$ and $-i_t+1$ are interchanged. So $inv(s_{\beta}v_{\underline{i},1})=inv(v_{\underline{i},1})+2$ and hence $l(s_{\beta}v_{\underline{i},1}) = l(v_{\underline{i},1})+1$.
  
  If $k=i_t$ with $|i_t-i_{t+1}|=2$, then $s_{\beta} v_{\underline{i},1} = v_{\underline{i},1}$ in $W^{I_r}$ since $s_{\beta} v_{\underline{i},1}(\omega_r) = v_{\underline{i},1}(\omega_r)$, a contradiction.
  
  
  
  If $\beta = \alpha_k+\alpha_{k+1}$ and $k \notin \{2, 3, i_t\}$ then since $s_\beta v_{\underline{i},1}(\omega_r) \geq v_{\underline{i},1}(\omega_r)$ we have $s_\beta v_{\underline{i},1} \leq v_{\underline{i},1}$ in $W^{I_r}$, a contradiction.
  
  If $k = 3$ then $\beta = \alpha_3+\alpha_4$. So  $s_{\beta} v_{\underline{i},1}(\omega_r) > v_{\underline{i},1}(\omega_r)$ and hence, $s_{\beta} v_{\underline{i},1} < v_{\underline{i},1}$, a contradiction.
  
   If $k=2$ then $\beta = \alpha_2+\alpha_3$, then $s_{\beta} v_{\underline{i},1}(\omega_r) = v_{\underline{i},1}(\omega_r)$ and hence $s_{\beta} v_{\underline{i},1} = v_{\underline{i},1}$ in $W^{I_r}$, a contradiction.
  
  If $j, j+1 \neq i_t$ then $s_{\beta}v_{\underline{i},1}(\omega_r) \geq v_{\underline{i},1}(\omega_r)$ and $s_{\beta}v_{\underline{i},1} \leq v_{\underline{i},1}$ in $W^{I_r}$, a contradiction.
  
  {\bf Case 3.} If $\beta = \alpha_{k-1}+\alpha_k$, a positive root of height $2$ the proof is similar to above case. 
  
    {\bf Case 4.} If $\beta = \alpha_1+\alpha_3$, then $s_\beta v_{\underline{i},1}=s_3s_{1}v_{\underline{i},1}$. From the reduced expression of $v_{\underline{i},1}$ we see that $l(s_\beta v_{\underline{i},1}) = l(v_{\underline{i},1})+2$, a contradiction.
  
  {\bf Case 5.} If $\beta = \alpha_1+\alpha_3+\alpha_4$ with $i_1 \geq 6$, then $s_\beta v_{\underline{i},1}(\omega_r) = \frac{1}{2}\alpha_1+\frac{1}{2}\alpha_2-\alpha_4+\sum \alpha_{i_k} < v_{\underline{i},1}(\omega_r)$. So $s_{\beta} v_{\underline{i},1} > v_{\underline{i},1}$. We show that $l(s_{\beta}v_{\underline{i},1})=l(v_{\underline{i},1})+1$.
  
We will prove this case for $m$ is odd. Note that $s_{\beta}v_{\underline{i}}=(-\underline{i},-3,-2,4,\underline{i^{\prime}},[-n,-5]\backslash\{[-\underline{i}],[-\underline{i^{\prime}}]\}\uparrow,1,-1,[5,n]\backslash\{[\underline{i}],[\underline{i^{\prime}}]\}\uparrow,-\underline{i^{\prime}},-4,2,3,\hat{\underline{i}})$. So $inv(s_{\beta}v_{\underline{i}})=inv(
 v_{\underline{i}})+4$ and $neg(s_{\beta}v_{\underline{i}})=neg(v_{\underline{i}})+2$. Hence $l(s_{\beta}v_{\underline{i}})=l(v_{\underline{i}})+1$. 
   
   Proof of the converse is clear from the definition of maximal $v$ such that $v(\omega_r) \geq 0$. 
 \end{proof}

   \begin{proposition}
 Let $G$ be of type $D_n$  and let $v,  v_{\underline{i}}, v_{\underline{i},1}, v_{\underline{i},2}, v_{2,\underline{i}}, v_{1,\underline{i}}$ be as defined in Proposition $5.1$. Let $w \in W^{I_r}$ be minimal such that $w(\omega_r) \leq 0$. Then the description of $w$ is the following:      
      
    (i) For $r=1$, $w=\left\{\begin{array}{lr}
    \
    (-n,-(n-2),\ldots,-4,-2,1,3,5,\ldots,n-3,n-1),& n\equiv 0(mod \,\, 4) \\
    (-n,-(n-2),\ldots,-4,-2,-1,3,5,\ldots,n-3,n-1),& n\equiv 2(mod \,\, 4) \\
    (-n,-(n-2),\ldots,-5,-3,-2,-1,4,\ldots,n-3,n-1),& n\equiv 1(mod \,\, 4) \\
    (-n,-(n-2),\ldots,-5,-3,-2,1,4,\ldots,n-3,n-1),& n\equiv 3(mod \,\, 4). 
    \end{array}\right.$

    (ii) For $r=2$, $w=\left\{\begin{array}{lr}
    \
    (n,-(n-2),\ldots,-4,-2,-1,3,5,\ldots,n-3,n-1),& n\equiv 0(mod \,\, 4) \\
    (n,-(n-2),\ldots,-4,-2,1,3,5,\ldots,n-3,n-1),& n\equiv 2(mod \,\, 4) \\
    (n,-(n-2),\ldots,-3,-2,1,4,6,\ldots,n-3,n-1),& n\equiv 1(mod \,\, 4) \\
    (n,-(n-2),\ldots,-3,-2,-1,4,6,\ldots,n-3,n-1),& n\equiv 3(mod \,\, 4). 
    \end{array}\right.$
    
    (iii) Let $3 \leq r \leq n-1$ and $(n+1)-r=2m$. 
    
     If $\underline{i} \in J_{m,[3,n]}$ then $w_{\underline{i}}=s_{i_1}s_{i_2}\ldots s_{i_m}v_{\underline{i}}$ with $w_{\underline{i}}(\omega_r)=-v_{\underline{i}}(\omega_r)$ and
     
     $w_{\underline{i}}=\left\{\begin{array}{lr}
     \
     (-1,[2,n]\backslash\{[\underline{i}],[\underline{i^{\prime}}]\}\uparrow,-\underline{i},\underline{i^\prime}),& \text{m is odd} \\
     (1,[2,n]\backslash\{[\underline{i}],[\underline{i^{\prime}}]\}\uparrow,-\underline{i},\underline{i^{\prime}}),& \text{m is even}. 
          \end{array}\right.$
    
     If $\underline{i} \in J_{m-1,[4,n]}$ then $w_{1,\underline{i}}=s_1s_{i_1}s_{i_2}\ldots s_{i_{m-1}}v_{1,\underline{i}}$ with $w_{1,\underline{i}}(\omega_r)=-v_{1,\underline{i}}(\omega_r)$ and
     
     $w_{1,\underline{i}}=\left\{\begin{array}{lr}
     \
     ([3,n]\backslash\{[\underline{i}],[\underline{i^{\prime}}]\}\uparrow,-\underline{i},
     -2,-1,\underline{i^\prime}),& \text{m is odd} \\
     (-t,[3,n]\backslash\{t,[\underline{i}],[\underline{i^{\prime}}]\}\uparrow,-\underline{i},
     -2,-1, \underline{i^\prime}),& \text{m is even}. 
     \end{array}\right.$ where $t=min{[3,n]\backslash\{[\underline{i}],[\underline{i^{\prime}}]\}}$.  
     
     If $\underline{i} \in J_{m-1,[4,n]}$ then $w_{2,\underline{i}}=s_2s_{i_1}s_{i_2}\ldots s_{i_{m-1}}v_{2,\underline{i}}$ with $w_{2,\underline{i}}(\omega_r)=-v_{2,\underline{i}}(\omega_r)$ and 
     
     $w_{2,\underline{i}}=\left\{\begin{array}{lr}
     \
     (-t,[3,n]\backslash\{t,[\underline{i}],[\underline{i^{\prime}}]\}\uparrow,-\underline{i},
     -2,1,\underline{i^\prime}),& \text{m is odd} \\
     ([3,n]\backslash\{[\underline{i}],[\underline{i^{\prime}}]\}\uparrow,-\underline{i},
     -2,1,\underline{i^\prime}),& \text{m is even}. 
     \end{array}\right.$ where $t=min{[3,n]\backslash\{[\underline{i}],[\underline{i^{\prime}}]\}}$.
       
     (iv) Let $3 \leq r \leq n-1$ and $(n+1)-r=2m+1$.

     For $\underline{i} \in J_{m,[4,n]}$ and $v_{\underline{i}}(\omega_r)=\frac{1}{2}\alpha_1+\frac{1}{2}\alpha_2+\sum_{k=1}^m \alpha_{i_k}$, we have $w_{\underline{i}}=s_1s_2s_{i_1}s_{i_2}\ldots s_{i_{m-1}}v_{\underline{i}}$. In this case $w_{\underline{i}}(\omega_r)=-v_{\underline{i}}(\omega_r)$ and
   
   $w_{\underline{i}}=\left\{\begin{array}{lr}
	\
	(1,[3,n]\backslash\{[\underline{i}],[\underline{i^{\prime}}]\}\uparrow,
	-\underline{i},-2,\underline{i^\prime}),& \text{m is odd} \\
	(-1,[3,n]\backslash\{[\underline{i}],[\underline{i^{\prime}}]\}\uparrow,
	-\underline{i},-2,\underline{i^\prime}),& \text{m is even}. 
\end{array}\right.$

  For $\underline{i} \in J_{m-1,[5,n]}$ and $v_{\underline{i},1}(\omega_r)=\displaystyle \frac{3}{2}\alpha_1+\frac{1}{2}\alpha_2+\alpha_3+\sum_{k=1}^{m-1} \alpha_{i_k}$, we have $w_{\underline{i},2}=s_2s_3s_1s_{i_1}s_{i_2}\ldots s_{i_{m-1}}v_{\underline{i},1}$.
  
  $w_{\underline{i},2}=\left\{\begin{array}{lr}
\
(4,[5,n]\backslash\{[\underline{i}],[\underline{i^{\prime}}]\}\uparrow,-\underline{i},
-3,-2,1,\underline{i^\prime}),& \text{m is odd} \\
(-4,[5,n]\backslash\{[\underline{i}],[\underline{i^{\prime}}]\}\uparrow,-\underline{i},
-3,-2,1,\underline{i^\prime}),& \text{m is even}. 
\end{array}\right.$
     
 For $\underline{i} \in J_{m-1,[5,n]}$ and $v_{\underline{i},2}(\omega_r)=\displaystyle \frac{1}{2}\alpha_1+\frac{3}{2}\alpha_2+\alpha_3+\sum_{k=1}^{m-1} \alpha_{i_k}$, we have $w_{\underline{i},1}=s_1s_3s_2s_{i_1}s_{i_2}\ldots s_{i_{m-1}}v_{\underline{i},2}$. \normalsize In these cases $w_{\underline{i},1}(\omega_r)=-v_{\underline{i},1} (\omega_r)$ and $w_{\underline{i},2}(\omega_r)=-v_{\underline{i},2}(\omega_r)$ and
 
 $w_{\underline{i},1}=\left\{\begin{array}{lr}
 \
 (-4,[5,n]\backslash\{[\underline{i}],[\underline{i^{\prime}}]\}\uparrow,-\underline{i},
 -3,-2,-1,\underline{i^\prime}),& \text{m is odd} \\
 (4,[5,n]\backslash\{[\underline{i}],[\underline{i^{\prime}}]\}\uparrow,-\underline{i},
 -3,-2,-1,\underline{i^\prime}),& \text{m is even}. 
 \end{array}\right.$
  
  (v) For $r=n$, we have $w=s_1s_2v$ and in this case $w(\omega_r)=-v(\omega_r)$. 
   \end{proposition}

   \begin{proof}
     For the proof of minimality of $w$ refer to \cite{KP}.
   \end{proof}
    
\begin{proposition} Let $v$, $w$, $v_{\underline{i}}$, $w_{\underline{i}}$, $v_{\underline{i},1}$, $w_{\underline{i},1}$, $v_{\underline{i},2}$, $w_{\underline{i},2}$, $v_{1,\underline{i}}$, $w_{1,\underline{i}}$, $v_{2, \underline{i}}$, $w_{2,\underline{i}}$ are as defined in Proposition 5.1 and Proposition 5.3. Let $w \in W^{P_r}$ be minimal and $v \in W^{P_r}$ be maximal such that $w(\omega_r) \leq 0$ and $v(\omega_r) \geq 0$. Then  $X_w^v \neq \emptyset $ iff the pair $(v,w)$ is one of the following:

(i) Let $3 \leq r \leq n-1$. For $(n+1)-r = 2m$:
    
    $(v,w) = \left\{\begin{array}{lr}
   \
    (v_{\underline{i}}, w_{\underline{j}}),& {s.t. |i_k-j_k|\leq 1 \,\, \forall \,\, 1 \leq k \leq m}\\
    (v_{\underline{i}},w_{1,\underline{j}}), (v_{\underline{i}},w_{2,\underline{j}}),& {s.t. |i_{k+1}-j_k|\leq 1 \,\, \forall \,\, 1 \leq k \leq m-1 \,\, with \,\, i_1 =3}\\
    (v_{1,\underline{i}},w_{\underline{j}}), (v_{2,\underline{i}},w_{\underline{j}}),& {s.t. |i_k-j_{k+1}|\leq 1 \,\, \forall \,\, 1 \leq k \leq m-1 \,\, with \,\, j_1 =3}.
  \end{array}\right.$
  
  For $(n+1)-r = 2m+1$:
    
    $(v,w) = \left\{\begin{array}{lr}
   \
    (v_{\underline{i}}, w_{\underline{j}}),& {s.t. |i_k-j_k|\leq 1 \,\, \forall \,\, 1 \leq k \leq m}\\
    (v_{\underline{i}},w_{\underline{j},1}),(v_{\underline{i}},w_{\underline{j},2})& {s.t. |i_{k+1}-j_k|\leq 1 \,\, \forall \,\, 1 \leq k \leq m-1 \,\, with \,\, i_1 =4}\\
    (v_{\underline{i},1},w_{\underline{j}}),(v_{\underline{i},2},w_{\underline{j}})& {s.t. |i_k-j_{k+1}|\leq 1 \,\, \forall \,\, 1 \leq k \leq m-1 \,\, with \,\, j_1 =4}\\
   (v_{\underline{i},1},w_{\underline{j},2}), (v_{\underline{i},2},w_{\underline{j},1})& {s.t. |i_k-j_k|\leq 1 \,\, \forall \,\, 1 \leq k \leq m-1}.
    \end{array}\right.$
   
(ii) For $r = 1, 2$ or $n$, $X_w^v \neq \emptyset $ for any $v$ and $w$.

    \end{proposition}
    
     \begin{proof} Let $X_w^v$ be non-empty. So, $v \leq w$. We prove this part for case (i), $(n+1)-r$ is odd and $(v,w) = (v_{\underline{i}}, w_{\underline{j},1})$. For other cases the proofs are similar. Let $(n+1)-r=2m+1$. We have $v_{\underline{i}}< w_{\underline{j},1}$. Now assume that $|i_{k+1}-j_k| \leq 1, \forall 1 \leq k \leq m-1$. We need to show that $i_1 = 4$. If not, then $i_1>4$. This implies that the coefficient of $\alpha_4$ in $v_{\underline{i}}(\omega_r)$ is zero. Since $l(w_{\underline{j},1})-l(v_{\underline{i}}) = m+2$ and $w_{\underline{j},1} > v_{\underline{i}}$ there exists $m+2$ positive roots $\beta$ such that $w_{\underline{j},1} = \displaystyle \prod_{card(\beta)=m+2}s_{\beta}v_{\underline{i}}$. On the other hand we have $v_{\underline{i}}(\omega_r)=\frac{1}{2}\alpha_1+\frac{1}{2}\alpha_2+0.\alpha_3+0.\alpha_4+\sum_{k=1}^m \alpha_{i_k}$ and $w_{\underline{j},1}(\omega_r)=-\frac{3}{2}\alpha_1-\frac{1}{2}\alpha_2-\alpha_3-\sum_{k=1}^{m-1}\alpha_{j_k}$. So, $w_{\underline{j},1}=[\displaystyle \prod_{t \in \{t:j_{t}=i_{t+1}\}}s_{\alpha_{i_{t+1}}}$ $ \displaystyle \prod_
     {t \in \{t:j_t=i_{t+1}-1\}\downarrow}s_{\alpha_{i_{t+1}}+\alpha_{i_
     {t+1}-1}}$ $ \displaystyle \prod_{t \in \{t:j_t=i_{t+1}+1\}\uparrow}s_{\alpha_{i_{t+1}}+\alpha_
     {i_{t+1}+1}}]s_1s_{3}s_{4}
    s_{3}s_{2}(s_{5}s_{6}\ldots s_{i_1-1}s_{i_1})(s_{4}s_{5}\ldots \\ s_{i_1-2}s_{i_1-1}) v_{\underline{i}}$. Note that the expression in the square bracket contains exactly $m-1$ reflections corresponding to $m-1$ distinct positive roots and it is independent of the word in between this expression and $v_{\underline{i}}$. So we have, $w_{\underline{j},1}=\displaystyle \prod_{card(\beta)>m+2} s_{\beta}v_{\underline{i}}$, a contradiction. So, $i_1 = 4$.


 Let $i_1=4$. We will show that $v_{\underline{i}} < w_{\underline{j},1}$ implies that $|i_{k+1}-j_k| \leq1,\forall 1 \leq k \leq m-1$. Since $l(w_{\underline{j},1})-l(v_{\underline{i}})=m+2$, by Lemma 5.2 we have $w_{\underline{j},1} = \displaystyle (\prod_{ \mathbf{card}\{\beta\} = m-1}s_\beta)s_1s_3s_4s_3s_2v_{\underline{i}}$ for positive roots $\beta$ such that $\beta$ is either $\alpha_{i_{t}}$ or $\alpha_{i_t}+\alpha_{{i_t}+1}$ or $\alpha_{i_t}+\alpha_{{i_t}-1}$ for some $i_t$, $2 \leq t \leq m$. So $w_{\underline{j},1}(\omega_r) = -\frac{3}{2}\alpha_1-\frac{1}{2}\alpha_2-\alpha_3\displaystyle -\sum_{{j_k}:|i_{k+1}-j_k| \leq 1} \alpha_{j_k}$. Hence, $|i_{k+1}-j_k| \leq 1$ $\forall 1 \leq k \leq m-1$.

     Conversely, let $v=v_{\underline{i}}, w=w_{\underline{j},1}$ with $|i_{k+1}-j_k| \leq 1 \,\, \forall \,\, 1 \leq k \leq m-1$ and $i_1 = 4$. Then we need to show that $v_{\underline{i}} < w_{\underline{j},1}$. Note that in this case $w_{\underline{j},1}=\displaystyle \prod_{ \{t:j_t=i_{t+1}\}}s_{\alpha_{i_{t+1}}}$ $ \displaystyle \prod_{ \{t:j_t=i_{t+1}-1\}\downarrow}s_{\alpha_{i_{t+1}}+\alpha_{i_{t+1}-1}}$ $ \displaystyle \prod_{ \{t:j_t=i_{t+1}+1\}\uparrow}s_{\alpha_{i_{t+1}}+\alpha_{i_{t+1}+1}}$ $(s_{\alpha_1}) (s_{\alpha_4+\alpha_3})(s_{\alpha_2})$ $v_{\underline{i}}$. We claim that in this product multiplication of each reflection to $v_{\underline{i}}$ amounts to increase the length by one and the product is greater than $v_{\underline{i}}$. Since $\underline{i} \in J_{m,[4,n]}$, $\underline{j} \in J_{m-1,[5,n]}$ and $|i_{k+1}-j_k|\leq 1 \forall 1 \leq k\leq m-1$ with $i_1=4$ we observe that the sets  $\{\bigcup\limits_{\{t:j_t=i_{t+1}+1\}}\{i_{t+1}-1,i_{t+1},
     i_{t+1}+1\}\}$, $\{\bigcup\limits_{\{t:j_t=i_{t+1}-1\}}\{i_{t+1}-2,i_{t+1}-1,
     i_{t+1}\}\}$ and  $\{\bigcup\limits_{\{t:j_t=i_{t+1}\}}\{i_{t+1}-1,i_{t+1}\}\}$ are mutually disjoint.
    
    It is easy to see that $s_{\alpha_1}s_{\alpha_3+\alpha_4}s_{\alpha_2}v_{\underline{i}} > v_{\underline{i}}$ and  $l(s_{\alpha_1}s_{\alpha_3+\alpha_4}s_{\alpha_2}v_{\underline{i}})=l( v_{\underline{i}})+3$.

     Now we claim that $s_{\alpha_1}
    s_{\alpha_3+\alpha_4}s_{\alpha_2}v_{\underline{i}}< \displaystyle \prod_{ \{t:j_t=i_{t+1}+1\}\uparrow}s_{\alpha_{i_{t+1}}+
    \alpha_{i_{t+1}+1}}s_{\alpha_1}
    s_{\alpha_3+\alpha_4}s_{\alpha_2}v_{\underline{i}}$ and the length is increased by the number of reflections multiplied. Let $t$ be maximal such that $j_t=i_{t+1}+1$. Since the sets $\{1,2,3,4\}$ and  $\{\cup_{\{t:j_t=i_{t+1}+1\}}\{i_{t+1}-1,i_{t+1},i_{t+1}+1\}\}$ are mutually disjoint, from Lemma $5.2$ we see that $s_{\alpha_1}s_{\alpha_3+\alpha_4}s_{\alpha_2}
    v_{\underline{i}}<s_{\alpha_{i_{t+1}+i_{t+1}+1}}s_{\alpha_1}
    s_{\alpha_3+\alpha_4}s_{\alpha_2}v_{\underline{i}}$. 
    
     Let $t$ be such that $i_{t}=i_{t+1}-2$ and $j_{t-1}=i_{t}+1$. We have $s_{\alpha_{i_t}+\alpha_{i_t+1}}s_{\alpha_{i_{t+1}}+\alpha_{i_{t+1}+1}} s_{\alpha_1} s_{\alpha_3+\alpha_4}\newline s_{\alpha_2}
  v_{\underline{i}}(\omega_r)=-\frac{3}{2}\alpha_1-\frac{1}{2}\alpha_2-\alpha_3+\displaystyle\sum_{k=2,\neq \{t, t+1\}}^m\alpha_{i_k}-\alpha_{i_{t+1}+1}-\alpha_{i_t+1}$. So $s_{\alpha_{i_t}+\alpha_{i_t+1}}s_{\alpha_{i_{t+1}}+\alpha_{i_{t+1}+1}} s_{\alpha_1}\newline s_{\alpha_3+\alpha_4}s_{\alpha_2}
  v_{\underline{i}}(\omega_r)<s_{\alpha_{i_{t+1}}+\alpha_{i_{t+1}+1}} s_{\alpha_1} s_{\alpha_3+\alpha_4}s_{\alpha_2}
  v_{\underline{i}}(\omega_r)$. Hence $s_{\alpha_{i_t}+\alpha_{i_t+1}}s_{\alpha_{i_{t+1}}+\alpha_{i_{t+1}+1}} s_{\alpha_1} s_{\alpha_3+\alpha_4} \\ s_{\alpha_2}
  v_{\underline{i}}>s_{\alpha_{i_{t+1}}+\alpha_{i_{t+1}+1}} s_{\alpha_1} s_{\alpha_3+\alpha_4}s_{\alpha_2}
  v_{\underline{i}}$. From the one line notations of these two elements we see that the length is increasing by $1$.
     
  
  Repeating this process we conclude that $v_{\underline{i}}<\displaystyle\prod_{t\in\{t:j_t=i_{t+1}\}}
  s_{\alpha_{i_{t+1}}}\prod_{t\in\{t:j_t=i_{t+1}-1\}
  \downarrow}s_{\alpha_{i_{t+1}}+\alpha_{i_{t+1}-1}}\newline
  \prod_{t\in\{t:j_t=i_{t+1}+1\}\uparrow}s_
       {\alpha_{i_{t+1}}+\alpha_{i_{t+1}+1}} s_{\alpha_1} s_{\alpha_3+\alpha_4}s_{\alpha_2}v_{\underline{i}}$.
        \end{proof}

 \begin{theorem} 
 Let $G$ be of type $D_n$ and let $P_r$ be the maximal parabolic subgroup corresponding to the simple root $\alpha_r$. Let $\mathcal{L}_r$ be the line bundle corresponding to the fundamental weight $\omega_r$.     
  Then $(X^{v}_{w})^{ss}_T(\mathcal{L}_r)$ is non-empty if and only if the pair $(v,w)$ is one of the following:  
  
  (i) For $3 \leq r \leq n-1$,\\  $(n+1)-r=2m$: \,\,  $(v,w) \,\, s.t.\left\{\begin{array}{lr}
   v \leq v_{\underline{i}} \,\, and \,\, w \geq w_{\underline{i}}\\
   v \leq v_{1,\underline{i}} \,\, and \,\, w \geq w_{1,\underline{i}}\\
   v \leq v_{2,\underline{i}} \,\, and \,\, w \geq w_{2,\underline{i}}.
   \end{array}\right.$
  
  $(n+1)-r=2m+1$: \,\,  $(v,w) \,\, s.t. \,\, \left\{\begin{array}{lr}
    v \leq v_{\underline{i}} \,\, and \,\, w \geq w_{\underline{i}}\\
     v \leq v_{\underline{i},1} \,\, and \,\, w \geq w_{\underline{i},2}\\
    v \leq v_{\underline{i},2} \,\, and \,\, w \geq w_{\underline{i},1}.
    \end{array}\right.$ 

(ii) For $r=1, 2$ and $n$, $(X^{v}_{w})^{ss}_T(\mathcal{L}_r)$ is non-empty for any $v$ and $w$.
    
    where $v$, $w$, $v_{\underline i}, w_{\underline i}, v_{\underline{i},1},  v_{\underline{i},2}, w_{\underline{i},1}$, $w_{\underline{i},2}$, $v_{1,\underline{i}}$, $w_{1,\underline{i}}$,$v_{2,\underline{i}}$ and $w_{2,\underline{i}}$ are as in Proposition 5.3.
    
    \end{theorem}
    
    \begin{proof} Now since $X_w^v \subseteq X_{w'}^{v'}$ implies  $(X^{v}_{w})^{ss}_T(\mathcal{L}_r) \subseteq (X^{v'}_{w'})^{ss}_T(\mathcal{L}_r)$, we can assume that $v$ is maximal and $w$ is minimal having the property that $v(\omega_r) \geq 0$ and $w(\omega_r) \leq 0$. We prove the theorem for case (i) and $(n+1)-r$ is odd. For other cases the proof is similar.
    
    Let $(n+1)-r=2m+1$. For each pair $(v,w)$, we construct a non-zero $T$-invariant section of $\mathcal{L}_r$ on $G/P_r$ which is not identically zero on  $X_w^v$.
    
     For $(v,w) = (v_{\underline{i}}, w_{\underline{i}})$ we have $v_{\underline{i}}(\omega_r)+w_{\underline{i}}(\omega_r) = 0$. So $p_{v_{\underline{i}}}p_{w_{\underline{i}}}$ is a non-zero $T$-invariant section of $\mathcal{L}_r$ on $G/P_r$ which is not identically zero on  $X_w^v$. 
     
     For $(v,w) = (v_{\underline{i},1},w_{\underline{i},2})$, we consider the sequence $v_{\underline{i},1} \leq s_1 v_{\underline{i},1} \leq s_{\alpha_{i_1}}s_{\alpha_{i_2}}\ldots s_{\alpha_{i_{m-1}}}s_3s_1v_{\underline{i},1} \\ \leq s_2s_{\alpha_{i_{1}}}s_{\alpha_{i_2}}\ldots s_{\alpha_{i_{m-1}}}s_3s_1v_{\underline{i},1} = w_{\underline{i},2}$. We have $v_{\underline{i},1}(\omega_r)+s_1 v_{\underline{i},1}(\omega_r)+s_{\alpha_{i_1}}s_{\alpha_{i_2}}\ldots s_{\alpha_{i_{m-1}}} s_3s_1 \\ v_{\underline{i},1}(\omega_r)+w_{\underline{i},2}(\omega_r) = 0$ and so $p_{v_{\underline{i},1}}p_{s_1 v_{\underline{i},1}}p_{s_{\alpha_{i_1}}s_{\alpha_{i_2}}\ldots s_{\alpha_{i_{m-1}}}s_3s_1v_{\underline{i},1}}p_{w_{\underline{i},2}}$ is a non-zero $T$-invariant section of $\mathcal{L}_r$ on $G/P_r$ which is not identically zero on  $X_w^v$.

     For $(v,w) = (v_{\underline{i},2},w_{\underline{i},1})$,  we consider the sequence $v_{\underline{i},2} \leq s_2 v_{\underline{i},2} \leq s_{\alpha_{i_1}}s_{\alpha_{i_2}}\ldots s_{\alpha_{i_{m-1}}}s_3 \\ s_2v_{\underline{i},2} \leq s_1s_{\alpha_{i_{1}}}s_{\alpha_{i_2}}\ldots s_{\alpha_{i_{m-1}}}s_3s_2v_{\underline{i},2} = w_{\underline{i},1}$. We have $v_{\underline{i},2}(\omega_r)+s_2 v_{\underline{i},2}(\omega_r)+s_{\alpha_{i_1}}s_{\alpha_{i_2}}\ldots s_{\alpha_{i_{m-1}}}s_3 
     \\ s_2v_{\underline{i},2}(\omega_r)+w_{\underline{i},1}(\omega_r) = 0$ and so $p_{v_{\underline{i},2}}p_{s_2 v_{\underline{i},2}}p_{s_{\alpha_{i_1}}s_{\alpha_{i_2}}\ldots s_{\alpha_{i_{m-1}}}s_3s_2v_{\underline{i},2}}p_{w_{\underline{i},1}}$  is a non-zero $T$-invariant section of $\mathcal{L}_r$ on $G/P_r$ which does not vanish identically zero on  $X_w^v$. 
  

   So, in all these cases we conclude that $(X{^{v}_{w}})^{ss}_T(\mathcal{L}_r) \neq \emptyset$.     
    
    Conversely, let $(X^v_w)^{ss}_T(\mathcal{L}_r)$ be non-empty. 
    
     Let $(v,w) = (v_{\underline{i},1},w_{\underline{j},2})$.   If $\underline{i} \neq \underline{j}$, then there exists $t$ such that $j_t \neq i_t$. Since $X^{v_{\underline{i},1}}_{w_{\underline{j},2}} \neq \emptyset$, by proposition 5.4 we have either $j_t = i_t+1$ or $j_t = i_t-1$. If $j_t=i_t+1$ then $w_{\underline{j},2}(\omega_r)$ = $-\frac{1}{2}\alpha_1-\frac{3}{2}\alpha_2-\alpha_3-\sum_{k=1,\neq t}^{m-1} {\alpha_{i_k}}-\alpha_{i_t+1}$ and if $j_t=i_t-1$ then $w_{\underline{j},2}(\omega_r)$ = $-\frac{1}{2}\alpha_1-\frac{3}{2}\alpha_2-\alpha_3-\sum_{k=1,\neq t}^{m-1} {\alpha_{i_k}}-\alpha_{i_t-1}$. Let $u \in W^{I_r}$ be such that $v_{\underline{i},1} \leq u \leq w_{\underline{j},2}$. Then $u$ is of the form $u = (\displaystyle \prod_{\beta} s_\beta) v_{\underline{i},1}$, where $\beta$'s are some positive roots. For $j_t=i_t+1$ at most one $\beta$ can be $\alpha_{i_t}+\alpha_{i_t+1}$ and none of the other $\beta$'s contain $\alpha_{i_t}$ or $\alpha_{i_t+1}$ as a summand. So in $u(\omega_r)$ the coefficient of $\alpha_{i_t}$ is either zero or one and the coefficient of $\alpha_{i_t+1}$ is either zero or $-1$. Similarly for $j_t=i_t-1$ at most one $\beta$ can be $\alpha_{i_t}+\alpha_{i_t-1}$ and none of the other $\beta$'s contain $\alpha_{i_t}$ or $\alpha_{i_t-1}$ as a summand. So in $u(\omega_r)$ the coefficient of $\alpha_{i_t}$ is either zero or one and the coefficient of  $\alpha_{i_t-1}$ is either zero or $-1$. For $j_t=i_t+1$, $u(\omega_r)$ contains either $\alpha_{i_t}$ or $\alpha_{i_t+1}$ as a summand and for $j_t=i_t-1$, $u(\omega_r)$ contains either $\alpha_{i_t}$ or $\alpha_{i_t-1}$ as a summand. So, there does not exist any sequence $v_{\underline{i},1} = u_1 \leq u_2 \leq \ldots \leq u_k = w_{\underline{j},2}$ such that $\sum_{l=1}^k u_l(\omega_r)=0$. So we don't have a nonzero $T$-invariant section which is not identically zero on $X_w^v$.

%

     Let $(v,w) = (v_{\underline{i}},w_{\underline{j},1})$ where $\underline{i}=(4,i_2,\ldots,i_m)$ and $\underline{j}=(i_2,\ldots,i_m)$. Then $v_{\underline{i}}(\omega_r)$ = $\frac{1}{2}\alpha_1+\frac{1}{2}\alpha_2+\alpha_4+\sum_{k=2}^{m} {\alpha_{i_k}}$ and $w_{\underline{j},1}(\omega_r)$ = $-\frac{3}{2}\alpha_1-\frac{1}{2}\alpha_2-\alpha_3-\sum_{k=2}^{m} {\alpha_{i_k}}$. Then any $u \in W^{I_r}$ such that $v_{\underline{i}} \leq u \leq w_{\underline{j},1}$ is of the form $u = (\displaystyle \prod_{\beta} s_\beta) v_{\underline{i}}$, where $\beta$'s are some positive roots. At most one $\beta$ can be $\alpha_3+\alpha_4$ and none of the other $\beta$'s can contain $\alpha_3$ or $\alpha_4$ as a summand. So, the coefficient of $\alpha_4$ in $u(\omega_r)$ is either zero or one and the coefficient of $\alpha_3$ in $u(\omega_r)$ is either zero or $-1$. So for any such $u$, $u(\omega_r)$ contains either $\alpha_3$ or $\alpha_4$ as a summand. So, in this case also there is no non zero $T$-invariant section which is not identically zero on $X_w^v$.

      Let $(v,w) = (v_{\underline{i}},w_{\underline{j},1})$ with $i_1 = 4$. If $(i_2,i_3, \ldots, i_m) \neq \underline{j}$, then there exists $t$ such that $j_t \neq i_{t+1}$. Since $X^{v_{\underline{i}}}_{w_{\underline{j},1}} \neq \emptyset$, by proposition 5.4 we have $j_t = i_{t+1}+1$ or $j_t = i_{t+1}-1$. If $j_t = i_{t+1}+1$ then $w_{\underline{j},1}(\omega_r)$ = $-\frac{3}{2}\alpha_1-\frac{1}{2}\alpha_2-\alpha_3-\sum_{k=1,\neq t}^{m-1} {\alpha_{i_{k+1}}}-\alpha_{i_{t+1}+1}$ and if $j_t=i_{t+1}-1$ then $w_{\underline{j},1}(\omega_r)$ = $-\frac{3}{2}\alpha_1-\frac{1}{2}\alpha_2-\alpha_3-\sum_{k=1,\neq t}^{m-1} {\alpha_{i_{k+1}}}-\alpha_{i_{t+1}-1}$. Then any $u \in W^{I_r}$ such that $v_{\underline{i}} \leq u \leq w_{\underline{j},1}$ is of the form $u = (\displaystyle \prod_{\beta} s_\beta) v_{\underline{i}}$, where $\beta$'s are some positive roots. For $j_t=i_{t+1}+1$ at most one $\beta$ can be $\alpha_{i_{t+1}}+\alpha_{i_{t+1}+1}$ and none of the other $\beta$'s contain $\alpha_{i_{t+1}}$ or $\alpha_{i_{t+1}+1}$ as a summand. So, in $u(\omega_r)$ the coefficient of $\alpha_{i_{t+1}}$ is either zero or one and the coefficient of $\alpha_{i_{t+1}+1}$ is either zero or $-1$. Similarly for $j_t=i_{t+1}-1$ at most one $\beta$ can be $\alpha_{i_{t+1}}+\alpha_{i_{t+1}-1}$ and none of the other $\beta$'s contain $\alpha_{i_{t+1}}$ or $\alpha_{i_{t+1}-1}$ as a summand. So, in $u(\omega_r)$ the coefficient of $\alpha_{i_{t+1}}$ is either zero or one and the coefficient of $\alpha_{i_{t+1}-1}$ is either zero or $-1$. For $j_t=i_{t+1}+1$, $u(\omega_r)$ contains either $\alpha_{i_{t+1}}$ or $\alpha_{i_{t+1}+1}$ as a summand and for $j_t=i_{t+1}-1$, $u(\omega_r)$ contains either $\alpha_{i_{t+1}}$ or $\alpha_{i_{t+1}-1}$ as a summand. So, like in previous cases here also we don't have a non zero $T$-invariant section which is not identically zero on $X_w^v$.  

For the pair $(v,w) = (v_{\underline{i}},w_{\underline{j}})$ with $\underline{i} \neq \underline{j}$ the proof is similar as in the cases in type $B$ and $C$. 
        \end{proof}
        
   We illustrate  Proposition 5.4 and Theorem 5.5 with an example.    

  {\bf Example:}  $D_5$, $\omega_3 = (\frac{3}{2},\frac{3}{2},3,2,1)$
 
  \hspace{2cm} $v_{\underline{i}}$ \hspace{2cm} $v_{\underline{i}}(\omega_3)$ \hspace{2cm} $w_{\underline{i}}(\omega_3)$ \hspace{2.5cm} $w_{\underline{i}}$
 
 $v_{(4)}$ = $(-1,5,-3,2,4)$ \hspace{.5cm} $(\frac{1}{2},\frac{1}{2},0,1,0)$ \hspace{.5cm} $(-\frac{1}{2},-\frac{1}{2},0,-1,0)$ \hspace{.5cm} $(1,5,-4,-2,3) = w_{(4)}$

$v_{(5)}$ = $(-1,3,-4,2,5)$ \hspace{.5cm} $(\frac{1}{2},\frac{1}{2},0,0,1)$ \hspace{.5cm} $(-\frac{1}{2},-\frac{1}{2},0,0,-1)$ \hspace{.5cm} $(1,3,-5,-2,4) = w_{(5)}$

$v_1$ = $(4,5,1,2,3)$ \hspace{1.5cm} $(\frac{3}{2},\frac{1}{2},1,0,0)$ \hspace{.5cm} $(-\frac{1}{2},-\frac{3}{2},-1,0,0)$ \hspace{.5cm} $(4,5,-3,-2,1) = w_2$

$v_2$ = $(-4,5,-1,2,3)$ \hspace{.75cm} $(\frac{1}{2},\frac{3}{2},1,0,0)$ \hspace{.5cm} $(-\frac{3}{2},-\frac{1}{2},-1,0,0)$ \hspace{.4cm} $(-4,5,-3,-2,-1) = w_1$

%
 So from the above observation, $X_{w_{(4)}}^{v_{(4)}}$, $X_{w_{(5)}}^{v_{(4)}}$, $X_{w_{(4)}}^{v_{(5)}}$, $X_{w_{(5)}}^{v_{(5)}}$, $X_{w_{1}}^{v_{(4)}}$, $X_{w_{2}}^{v_{(4)}}$, $X_{w_{(4)}}^{v_{1}}$, $X_{w_{(4)}}^{v_{2}}$, $X_{w_{2}}^{v_{1}}$ and $X_{w_{1}}^{v_{2}}$ are all non-empty. We have $(X_{w_{(4)}}^{v_{(5)}})^{ss}_T(\mathcal{L}_3)$, $(X_{w_{(5)}}^{v_{(4)}})^{ss}_T(\mathcal{L}_3)$, $(X_{w_{1}}^{v_{(4)}})^{ss}_T(\mathcal{L}_3)$, $(X_{w_{(4)}}^{v_{1}})^{ss}_T(\mathcal{L}_3)$, $(X_{w_{(4)}}^{v_{2}})^{ss}_T(\mathcal{L}_3)$, and $(X_{w_{2}}^{v_{(4)}})^{ss}_T(\mathcal{L}_3)$ are empty whereas $(X_{w_{(4)}}^{v_{(4)}})^{ss}_T(\mathcal{L}_3)$, $(X_{w_{(5)}}^{v_{(5)}})^{ss}_T(\mathcal{L}_3)$, $(X_{w_{1}}^{v_{2}})^{ss}_T(\mathcal{L}_3)$ and $(X_{w_{2}}^{v_{1}})^{ss}_T(\mathcal{L}_3)$ are non-empty.

  \end{document}